\documentclass[final]{dmtcs_episciences}
\usepackage[utf8]{inputenc}
\usepackage{subfigure}
\usepackage[T1]{fontenc}
\usepackage{tempora}
\usepackage[utf8]{inputenc}
\usepackage{stmaryrd}
\usepackage{a4wide}
\usepackage{amsmath,amsthm,amssymb}                                                                      
\usepackage{tcolorbox}                                        \usepackage{comment}      
\usepackage{mathrsfs}
\usepackage{listings}
\usepackage{hyperref}
\usepackage{graphicx}
\usepackage{tikz}
\usetikzlibrary{fit}
\usepackage{bbm}
\usepackage{bm}
\usepackage{xcolor}
\usepackage{mathdots}
\usepackage{float}
\allowdisplaybreaks
\newcommand{\N}{\mathbb{N}}

\newcommand{\vect}[1]{\boldsymbol{#1}}
\usepackage{graphicx}
\usepackage{pgfplots}
\usepackage{pgfplotstable}
\usepgfplotslibrary{fillbetween}
\usepackage{pgfplots}
\pgfplotsset{compat=1.5.1}
\usepackage{float}
\theoremstyle{definition}

\theoremstyle{plain}
\newtheorem{theorem}{Theorem}
\newtheorem{definition}[theorem]{Definition}
\newtheorem{lemma}[theorem]{Lemma}
\newtheorem{proposition}[theorem]{Proposition}

\newtheorem{claim}[theorem]{Claim}
\newtheorem{remark}[theorem]{Remark}
\newtheorem{question}{Question}

\usepackage[round]{natbib}

\author[Alex Alochukwu et al.]{  Alex Alochukwu\affiliationmark{1}\thanks{This work is an output of the American Mathematical Society's Mathematics Research Community (MRC): Trees in Many Contexts. }\and Audace Dossou-Olory \affiliationmark{2}
  \and Fadekemi Janet Osaye\affiliationmark{3}\\
  \and  Valisoa R.~M.~Rakotonarivo\affiliationmark{4} \and Shashank Ravichandran\affiliationmark{5}\\\and Sarah J.~Selkirk\affiliationmark{6} \and Hua Wang\affiliationmark{7}\and Hays Whitlatch\affiliationmark{8}}
\title[Characterization of Trees with Maximum Security]{Characterization of Trees with Maximum Security}
\affiliation{
Department of Mathematics, Computer Science and Physics, Albany State University, Albany, United States of America.\\
Institut National de l'Eau, 
and Centre d'Excellence d'Afrique pour l'Eau et
l'Assainissement and Institut de Mathématiques et de Sciences Physiques, Université d'Abomey-Calavi, Abomey-Calavi, Bénin.\\
  Department of Mathematics and Statistics, Troy University, Troy, United States of America.\\
  University of Pretoria, Pretoria, South Africa.\\
  Department of Mathematical Sciences, University of Delaware, Newark, United States of America.\\
  Institut
  f\"ur Mathematik, Alpen-Adria-Uni\-ver\-si\-t\"at Klagenfurt,
  Klagenfurt am Wörthersee, Austria; Centre for Discrete Mathematics and its Applications (DIMAP) \&
  Department of Computer Science, University of Warwick,
  Coventry, United Kingdom.\\
  Georgia Southern University, Statesboro, United States of America.\\
  Department of Mathematics, Gonzaga University, 
Spokane, United States of America.}
\begin{document}
\publicationdata{vol. 28:2}{2026}{12}{10.46298/dmtcs.14872}{2024-12-02; 2024-12-02; 2025-10-28; 2026-01-21}{2026-02-05}

\maketitle
\begin{abstract}
   The rank (also known as protection number or leaf-height) of a vertex in a rooted tree is the minimum distance between the vertex and any of its leaf descendants. We consider the sum of ranks over all vertices (known as the security) in proper binary trees, and produce a classification of families of proper binary trees for which the security is maximized. In addition, extremal results relating to the maximum rank among all vertices in families of trees are discussed
\end{abstract}



\section{Introduction}
\label{sec:intro}

A large body of research in graph theory is devoted to distance in graphs. Some such examples are the diameter 
(greatest distance between any pair of vertices), domination (restriction of distance from a set of vertices), 
and independence (restriction of distance between elements of a set of vertices). A particular distance parameter 
and centrality measure, the \emph{eccentricity} of a vertex $v$, is defined to be the maximum distance between $v$ and 
any other vertex in the graph. Among all vertices, one with minimum eccentricity is called the \emph{center} of 
the graph. Formally, given a graph $G = (V, E)$ and a vertex $v \in V$, the eccentricity of $v$ is given by 
\begin{equation*}
    \mathsf{ecc}(v) = \max_{u \in V}d(u, v).
\end{equation*}
When vertices and edges in $G$, for example, represent customers and road networks, respectively, a vertex with minimum eccentricity (center) would be an ideal location to place a store or warehouse aiming to service the customers. Therefore, 
eccentricity has applications, among others, in supply chain management. 

We now consider a version of ``anti-eccentricity'' of a vertex, taking the minimum instead of maximum. However, in the context of graphs, the minimum distance among all vertices is trivially $0$. We therefore consider the graph family of trees, and the minimum distance to a leaf. This measure of ``anti-eccentricity'' is known as \emph{rank}, \emph{protection number}, or \emph{leaf-height} and usually studied in rooted trees. The history and formal definitions of this parameter are given below. 

A vertex in a rooted tree is called \emph{protected} if it is neither a leaf nor the 
neighbour of a leaf. This concept was first studied in~\cite{cheon}, and later the
number of protected vertices in families of trees including $k$-ary trees, digital and 
binary search trees, tries, and suffix trees was 
determined~\cite{Devroye-Janson:2014:protec,Du-Prodinger:2012:notes,Gaither-Homma-Sellke-Ward:2012:protec,
Holmgren-Janson:2015:asymp,Mahmoud-Ward:2015:asymp,Mahmoud-Ward:2012:protec,Mansour:2011:protec}. 
The definition of protected vertices was then generalized to \emph{$\ell$-protected vertices}: 
vertices at a distance of at least $\ell$ away from any leaf descendant. With this terminology, a 
protected vertex is equivalent to a 2-protected vertex. The number of $\ell$-protected vertices 
has also been studied in a wide range of trees~\cite{Bona:2014, Bona-Pittel:2017, Copenhaver:2017:protec,
Devroye-Janson:2014:protec, Heuberger-Prodinger:2017:protec-number-plane-trees}. This leads to the 
formal definition of the protection number, examples of which can be seen in Figure~\ref{fig:rank-examples}.

\begin{definition}[Protection number]
    Let $T$ be a rooted tree, $v$ be a vertex in $T$, $T(v)$ the subtree induced by $v$ and its descendants, and $L(T(v))$ its set of leaves. The \emph{protection number} $R_T(v)$ of $v$ is given by 
    \begin{equation*}
        R_T(v) = \min_{u \in L(T(v))}d(u, v).
    \end{equation*} 
\end{definition}

\begin{figure}[H]
    \centering

    \begin{tikzpicture}[scale = 0.9]
        \node[circle, draw] (0) at (0, 0){2};
        \node[circle, draw] (1) at (-2, -2/2){1};
        \node[circle, draw] (2) at (2, -2/2){1};
        \node[circle, draw] (3) at (-3, -4/2){0};
        \node[circle, draw] (4) at (-1, -4/2){1};
        \node[circle, draw] (5) at (1, -4/2){0};
        \node[circle, draw] (6) at (3, -4/2){2};
        \node[circle, draw] (7) at (-1.5, -6/2){0};
        \node[circle, draw] (8) at (-0.5, -6/2){0};
        \node[circle, draw] (9) at (2, -6/2){1};
        \node[circle, draw] (10) at (4, -6/2){1};
        \node[circle, draw] (11) at (1.5, -8/2){0};
        \node[circle, draw] (12) at (2.5, -8/2){0};
        \node[circle, draw] (13) at (3.5, -8/2){0};
        \node[circle, draw] (14) at (4.5, -8/2){0};

        \draw[-] (0) to (1);
        \draw[-] (0) to (2);
        \draw[-] (1) to (3);
        \draw[-] (1) to (4);
        \draw[-] (4) to (7);
        \draw[-] (4) to (8);
        \draw[-] (2) to (5);
        \draw[-] (2) to (6);
        \draw[-] (6) to (9);
        \draw[-] (6) to (10);
        \draw[-] (9) to (11);
        \draw[-] (9) to (12);
        \draw[-] (10) to (13);
        \draw[-] (10) to (14);
    \end{tikzpicture}

    \caption{A binary tree with the protection number of each vertex marked on the vertex.}
    \label{fig:rank-examples}
\end{figure}
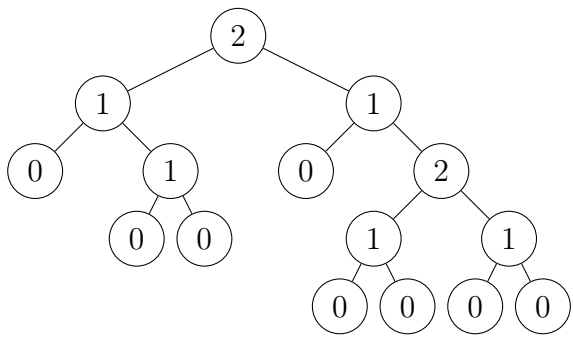

Studies on the protection number of the root of a tree in specific families of trees as well as the general context of simply generated trees were completed in~\cite{Gittenberger-Golebiewski-Larcher-Sulkowska:2021:protec,Golebiewski-Klimczak:2019:protec,Heuberger-Prodinger:2017:protec-number-plane-trees}. Most recently, extremal parameters such as the maximum protection number of a tree have been studied using powerful probabilistic and analytic techniques~\cite{Devroye-Goh-Zhao:2023:peel, Heuberger-Selkirk-Wagner:2023:protec}. This work serves to further continue this trend, by providing constructive solutions to extremal problems of this variety. 

It is easy to imagine a range of applications for extremal parameters relating to the protection number. Depending on what the leaves represent, one may wish to leave as many vertices unprotected (if leaves represent customers that need to be easily contacted, for instance) or to have as many vertices protected as possible (if leaves represent points of entry for hackers, for instance). 
Trees with degree constraints, binary trees, in particular, have been shown to play a significant role in such applications.

From this point onwards, we will mainly use the term ``rank'' instead of ``protection number'', since it reduces the phrase length and, thus, the ease of parsing phrases in several instances\footnote{We have predominantly used ``protection number'' until now to ensure that this work connects to the other bodies of work on this topic.}.  

With current literature in mind, we consider the rank from a constructive perspective: 
\begin{enumerate}
\item find trees with a given number of vertices that maximize or minimize the number of $\ell$-protected vertices for a given $\ell$; or \label{item:1}
\item find trees with a maximum or minimum sum of ranks of its vertices, given the number of vertices. \label{item:2} 
\end{enumerate}

As may be expected, the solutions to these two questions are in general identical. 
For general rooted trees it is not hard to see (as shown in Section~\ref{sec:most-protected}) 
that the path is the ``most protected'' among trees of order $n$ while the star $S_n$ is the ``least protected''. This may not be surprising as the path and the star are, respectively, the sparsest and densest tree among general trees of the same order, and they have been shown to be the extremal trees with respect to many different graph invariants. We aim to extend the study of the least and most protected trees to rooted proper $k$-ary trees of given order and trees with a given outdegree sequence. 

Our main contribution is the classification of proper binary trees which achieve~(\ref{item:2}), and thus also~(\ref{item:1}). For this, we introduce the following terminology and notation. 

\begin{definition}[Security]
    Let $T$ be a rooted tree. The \emph{security} of $T$, denoted $R(T)$, is given by 
    \begin{equation*}
        R(T) = \sum_{v \in V(T)} R_T(v).
    \end{equation*}
    Additionally, the number of $\ell$-protected vertices in $T$ is denoted by $n_{\ell}(T)$.
\end{definition}

For the binary tree $T$ in Figure~\ref{fig:rank-examples}, the security is given by $R(T) = 0\cdot 8 + 1\cdot 5+ 2\cdot 2 = 9$. We seek trees that maximize or minimize $R(T)$ and $n_{\ell}(T)$ for any $\ell$. 

The layout of this paper is as follows. In Section~\ref{sec:binary-power} we provide, among other things, a proof that a family of binary trees obtains maximum security. This is done via ``switching lemmata'' which show that certain ``switching'' operations do not decrease the security of a tree. As a consequence, we also provide a formula for the maximum $R(T)$. Then, in Section~\ref{sec:almost-complete} we provide an additional classification of a family of trees that obtain maximum security. This is done by showing that these trees have equal security to those already proved to have maximum security in Section~\ref{sec:binary-power}. It is interesting to note that this additional family of trees has appeared in many studies of other extremal problems on binary trees. To conclude, we provide some extremal results relating to the maximum rank in different families of trees in Section~\ref{sec:most-protected}.

\section{Most protected binary trees: From binary representation}\label{sec:binary-power}

Proper binary trees, which in our context refer to unordered rooted trees where each internal vertex has exactly two children,
are one of the most commonly studied structures when studying ranks. We start by considering trees that maximize the security, $\max{R(T)}$, among all proper binary trees on $\ell$ leaves. Note that the trees achieving the maximum for the total rank need not be unique, as seen in Figure \ref{fig:maxbin}. Indeed, for certain values of  $\ell$, the tree attaining the maximum security is unique, while for other values of  $\ell$, it is not. In this section, we will give a characterization of one of the optimal proper binary trees. 

\begin{figure}[H]
    \centering
    	\begin{tikzpicture}[nodes={fill=black,circle,inner sep=1.4pt}, sibling distance=2cm,scale=0.6]
		\node{}[sibling distance=3cm]
		child { node {} [sibling distance=1.5cm]
			child { node{} [sibling distance=1cm]
					child { node{}}
					child { node{}}}
			child { node {}[sibling distance=1cm]
					child { node{}}
					child { node{}}}
		}
		child { node{}[sibling distance=1.5cm]
			child { node {} [sibling distance=1cm]
							child { node{}}
							child { node{}}	}
			child { node {} }};

	\node at (6,0) {}[sibling distance=3cm]
	child { node {} [sibling distance=1.5cm]
		child { node{} [sibling distance=1.5cm]
			child { node{}[sibling distance=1cm]
				child{node{}}
				child{node{}}}
			child { node{}}[sibling distance=1cm]
				child{node{}}
				child{node{}}}
		child { node {}}
	}
	child { node{}[sibling distance=1.5cm]
		child { node {} }
		child { node {} }};

	\node at (12,0){}[sibling distance=3cm]
	child { node {} [sibling distance=2cm]
		child { node{} [sibling distance=1.5cm]
			child { node{}[sibling distance=1cm]
				child{node{}}
				child{node{}}}
			child { node{}}[sibling distance=1cm]
				child{node{}}
				child{node{}}}
		child { node {}[sibling distance=1.5cm]
			child{node{}}
			child{node{}}}
	}
	child { node{}};

    \node at (18,0){}[sibling distance=3cm]
	child { node {} [sibling distance=2cm]
		child { node{} [sibling distance=1.5cm]
			child { node{}[sibling distance=1cm]
				child{node{}}
				child{node{}}}
			child { node{}}}
		child { node {}[sibling distance=1.5cm]
			child{node{}}
			child{node{}}}
	}
	child { node{} [sibling distance=2cm]
                child{node{}}
                child{node{}}};
\end{tikzpicture}
    \caption{All four non-isomorphic proper binary trees which achieve the maximum total rank of 8 among all proper binary trees with 7 leaves.}
    \label{fig:maxbin}
\end{figure}
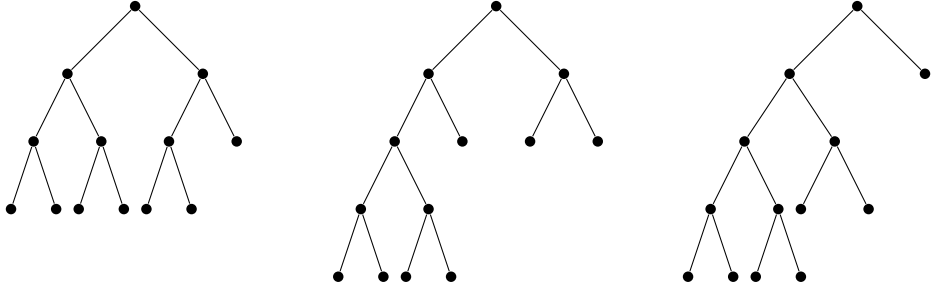

Let $\vect{L}:=(n_1, n_2, \ldots, n_k)$ correspond to the \emph{binary power representation} of 
$$ \ell = \sum_{i=1}^{k} 2^{n_i} $$
with $n_1 > n_2 > \cdots > n_k \geq 0$. We will show that $T_{\vect{L}}$, defined below, is a tree for which the value of $R(T)$ is maximized. Intuitively, $T_{\vect{L}}$ is the binary caterpillar tree on $k$ leaves $u_1, \ldots, u_k$, with $u_1$ being the deepest leaf and $u_k$ the leaf incident to the root, with each leaf $u_i$ identified with the root of a complete binary tree of height $n_i$. Formally, this is described below. 

\begin{definition}\label{def:T-L}
For a fixed integer $\ell \geq 2$ with binary power representation $\vect{L} = (n_1, n_2, \ldots, n_k)$, the binary tree $T_{\vect{L}}$ is rooted at a vertex $v_k$ and constructed by:
\begin{itemize}
\item identifying the root of a complete binary tree on $2^{n_1}$ leaves with one end $v_1$ of a path $P=v_1v_2\ldots v_k$;
\item using an edge to join the root $u_i$ of a complete binary tree on $2^{n_i}$ leaves with vertex $v_i$ (for $2 \leq i \leq k$).
\end{itemize}
\end{definition}

As an example, for $\ell = 11 = 2^3 + 2^1 + 2^0$ the binary power representation is $\vect{L} = (3,1,0)$, and $T_{\vect{L}}$ is shown in Figure~\ref{fig:tl}, where the vertices $v_i$ and $u_i$ are labelled. 

\begin{figure}[H]
\centering
    \begin{tikzpicture}[scale=.6]
        \node[fill=black,circle,inner sep=1.4pt] (t1) at (0,0) {};
        \node[fill=black,circle,inner sep=1.4pt] (t2) at (1,1) {};
        \node[fill=black,circle,inner sep=1.4pt] (t3) at (2.5,0) {};
        \node[fill=black,circle,inner sep=1.4pt] (t4) at (2,2) {};
        \node[fill=black,circle,inner sep=1.4pt] (t5) at (3,1) {};
        \node[fill=black,circle,inner sep=1.4pt] (t5') at (4,0) {};
		\node[fill=black,circle,inner sep=1.4pt] (t6) at (3,3) {};
        \node[fill=black,circle,inner sep=1.4pt] (t7) at (4,2) {};
        \node[fill=black,circle,inner sep=1.4pt] (t8) at (4,4) {};
        \node[fill=black,circle,inner sep=1.4pt] (t9) at (5,3) {};
        \node[fill=black,circle,inner sep=1.4pt] at (1.5,0) {};
        \node[fill=black,circle,inner sep=1.4pt] at (-0.4,-1) {};
        \node[fill=black,circle,inner sep=1.4pt] at (0.4,-1) {};
        \node[fill=black,circle,inner sep=1.4pt] at (1.1,-1) {};
        \node[fill=black,circle,inner sep=1.4pt] at (1.8,-1) {};
        \node[fill=black,circle,inner sep=1.4pt] at (2.2,-1) {};
        \node[fill=black,circle,inner sep=1.4pt] at (2.9,-1) {};
        \node[fill=black,circle,inner sep=1.4pt] at (3.6,-1) {};
        \node[fill=black,circle,inner sep=1.4pt] at (4.4,-1) {};
        \node[fill=black,circle,inner sep=1.4pt] at (3.5,1) {};
        \node[fill=black,circle,inner sep=1.4pt] at (4.5,1) {};

        \draw (t4)--(4,0);
        \draw (1,1)--(1.5,0);
        \draw (3,1)--(2.5,0);
        \draw (t8)--(0,0);
        \draw (t7)--(t6);
        \draw (t8)--(t9);
        \draw (t7)--(4.5,1);
        \draw (t7)--(3.5,1);
        \draw (0,0)--(-.4,-1);
        \draw (0,0)--(.4,-1);
        \draw (0+1.5,0)--(-.4+1.5,-1);
        \draw (0+1.5,0)--(.3+1.5,-1);
        \draw (0+2.5,0)--(-.3+2.5,-1);
        \draw (0+2.5,0)--(.4+2.5,-1);
        \draw (0+4,0)--(-.4+4,-1);
        \draw (0+4,0)--(.4+4,-1);

        \node at (4,4.3) {$v_3$};
        \node at (2.8,3.3) {$v_2$};
        \node at (1,2.3) {$u_1 = v_1$};
        \node at (5.1,3.3) {$u_3$};
        \node at (4.1,2.3) {$u_2$};

        \end{tikzpicture}
    \caption{A proper binary tree $T_{\vect{L}}$ on $\ell = 11$ leaves that maximizes $R(T)$.}
    \label{fig:tl}
\end{figure}
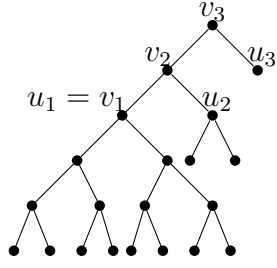

The main result of this section states that $T_{\vect{L}}$ is one of the optimal proper binary trees that maximizes $R(T)$.

\begin{theorem}\label{theo:tl}
Given a proper binary tree $T$ with $\ell$ leaves, where $\ell$ has a binary power representation of $\vect{L}$, it holds that $R(T) \leq R(T_{\vect{L}})$.
\end{theorem}

The proof of this theorem will depend on several other results which will be proved throughout the remainder of this section.

\subsection{Partition of a proper binary tree}
Before we can prove the optimality of $T_{\vect{L}}$, we first introduce a ``partition'' of any given proper binary tree on $\ell$ leaves that would allow us to compare them.

\begin{definition}
Given a proper binary tree $T$, we say a vertex $v$ is \emph{saturated} if $T(v)$ is a complete binary tree but $T(u)$ is not complete for any of $v$'s ancestors $u$. Let $v_{1}$, \ldots, $v_{j}$ be the saturated vertices in $T$. Recording $|L(T(v_{i}))| = 2^{m_i}$ for $1\leq i \leq j$, we have that the number of leaves $\ell$ in $T$, is given by
$$\ell = \sum_{i=1}^{j} 2^{m_i}, $$
where without loss of generality, $m_1 \geq m_2 \geq \cdots \geq m_j \geq 0$. We call $\vect{M}=(m_1, m_2, \ldots , m_j)$ the corresponding \emph{partition vector}. 
\end{definition}

Note that the proper binary tree corresponding to a given $\vect{M}$ is not necessarily unique, as shown in Figure~\ref{fig:tm}, where the saturated vertices are highlighted.

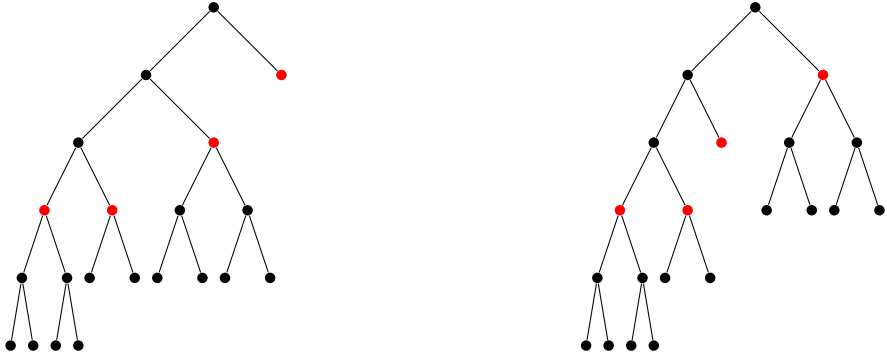
\begin{figure}[H]
\centering
    \begin{tikzpicture}
[nodes={fill=black,circle,inner sep=1.4pt}, sibling distance=3cm, scale=0.6]
		\node{}
		child { 
			node {} [sibling distance=3cm]
			child { node{} [sibling distance=1.5cm]
					child { node [fill=red] {}[sibling distance=1cm]
						 child { node {} [sibling distance=0.5cm]
						  	child { node {} }
						  	child { node {} }}
						 child { node {} [sibling distance=0.5cm]
						 	child { node {} }
					 		child { node {} } }}
					child { node [fill=red]{} [sibling distance=1cm]
						child { node {} }
						child { node {} }}}
			child { node [fill=red] {}[sibling distance=1.5cm]
					child { node {}[sibling distance=1cm]
						child { node {} }
						child { node {} } }
					child { node {} [sibling distance=1cm]
						child { node {} }
					    child { node {} }} }
		}
		child { node[fill=red] {}};
		
	\node at (12,0){}
		child { 
			node {} [sibling distance=1.5cm]
			child { node{} [sibling distance=1.5cm]
				child { node [fill=red] {}[sibling distance=1cm]
					child { node {} [sibling distance=0.5cm]
						child { node {} }
						child { node {} }}
					child { node {} [sibling distance=0.5cm]
						child { node {} }
						child { node {} } }}
				child { node [fill=red]{} [sibling distance=1cm]
					child { node {} }
					child { node {} }}}
			child { node [fill=red] {}  }
		}
		child { node[fill=red]{}{ [sibling distance=1.5cm]
			child { node {} [sibling distance=1cm]
					child{node{}}
					child{node{}}
					}
			child { node {} [sibling distance=1cm]
					child{node{}}
					child{node{}}}}
		 };

	\end{tikzpicture}
\caption{Two different proper binary trees corresponding to $\vect{M}=(2,2,1,0)$ with security 15 and 14, respectively.}\label{fig:tm}
\end{figure}

When a non-root vertex $v$ is saturated, we denote its parent by $v_0$ and its sibling by $v_1$ unless otherwise noted. Furthermore, in the tree $T_{\vect{L}}$, all saturated vertices have different ranks (corresponding to the different $n_i$).

\subsection{``Switching'' Lemmata}
In order to compare $T$ (corresponding to $\vect{M}$) and $T_{\vect{L}}$, we introduce some useful operations that allow us to transform from $T$ to $T_{\vect{L}}$ without ever decreasing security $R(\cdot)$, ultimately switching $\vect{M}$ (which may have repeated entries) to $\vect{L}$, where $\vect{L}$ has no repeated entries. These operations involve the application of lemmata that will be proved in this section.
\begin{lemma}\label{lem:sw1}
Let $T$ be a proper binary tree with saturated vertices $u$ and $w$ such that $R_T(u) = R_T(w)$. Let $u$ and $w$'s respective siblings be $u_1$, $w_1$, parents be $u_0$, $w_0$, with neither $u_0$ or $w_0$ an ancestor of the other. If $R_T(w_1) \geq R_T(u_1)$ and $T^* = T - uu_0 - w_1w_0 + uw_0 + w_1u_0$ (see Figure~\ref{fig:sw1}), then $R(T^*) \geq R(T)$.
\end{lemma}

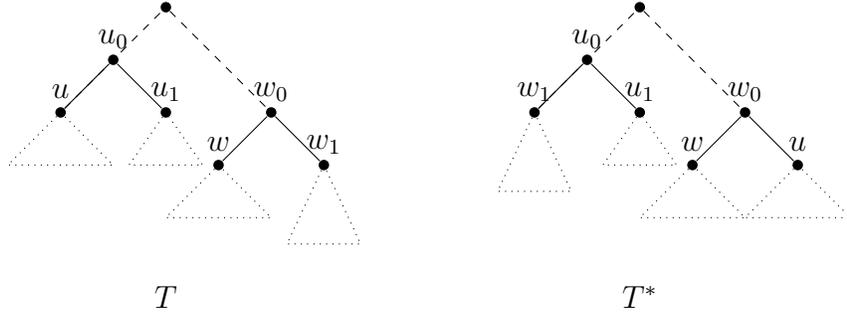
\begin{figure}[H]
\centering
    \begin{tikzpicture}[scale=.7]
        
        \node[fill=black,circle,inner sep=1.4pt] (t2) at (1,1) {};
        \node[fill=black,circle,inner sep=1.4pt] (t4) at (2,2) {};
        \node[fill=black,circle,inner sep=1.4pt] (t5) at (3,1) {};
		\node[fill=black,circle,inner sep=1.4pt] (t6) at (3,3) {};
       
        \node[fill=black,circle,inner sep=1.4pt] (t1) at (5,1) {};
        \node[fill=black,circle,inner sep=1.4pt] (t8) at (4,0) {};
        \node[fill=black,circle,inner sep=1.4pt] (t9) at (6,0) {};

        \draw [dashed] (t2)--(t6)--(t1);
        \draw (t2)--(t4);
        \draw [dotted] (t2)--(0,0)--(2,0)--(t2);
        \draw [dotted] (t5)--(2.3,0)--(3.7,0)--(t5);
        \draw (t8)--(t1)--(t9);
        \draw [dotted] (t8)--(3,-1)--(5,-1)--(t8);
        \draw [dotted] (t9)--(5.3,-1.5)--(6.7,-1.5)--(t9);
        \draw (t4)--(t5);
        
        \node at (2,2.4) {$u_0$};
        \node at (1,1.4) {$u$};
        \node at (3,1.4) {$u_1$};
        
        \node at (5,1.4) {$w_0$};
        \node at (4,.4) {$w$};
        \node at (6,.4) {$w_1$};
        
        \node at (3, -2.5) {$T$};

        \begin{scope}[shift={(+9,0)}]
            \node[fill=black,circle,inner sep=1.4pt] (t2) at (1,1) {};
            \node[fill=black,circle,inner sep=1.4pt] (t4) at (2,2) {};
            \node[fill=black,circle,inner sep=1.4pt] (t5) at (3,1) {};
			\node[fill=black,circle,inner sep=1.4pt] (t6) at (3,3) {};
       
            \node[fill=black,circle,inner sep=1.4pt] (t1) at (5,1) {};
            \node[fill=black,circle,inner sep=1.4pt] (t8) at (4,0) {};
            \node[fill=black,circle,inner sep=1.4pt] (t9) at (6,0) {};

            \draw [dashed] (t2)--(t6)--(t1);
            \draw (t2)--(t4);
            \draw [dotted] (t2)--(.3,-.5)--(1.7,-.5)--(t2);
            \draw [dotted] (t5)--(2.3,0)--(3.7,0)--(t5);
            \draw (t8)--(t1)--(t9);
            \draw [dotted] (t8)--(3,-1)--(5,-1)--(t8);
            \draw [dotted] (t9)--(5,-1)--(7,-1)--(t9);
            \draw (t4)--(t5);
            \node at (2,2.4) {$u_0$};
            \node at (1,1.4) {$w_1$};
            \node at (3,1.4) {$u_1$};
            
            \node at (5,1.4) {$w_0$};
            \node at (4,.4) {$w$};
            \node at (6,.4) {$u$};
            
            \node at (3, -2.5) {$T^*$};
        \end{scope}

    \end{tikzpicture}
    \caption{The trees $T$ (on the left) and $T^*$ (on the right) from Lemma \ref{lem:sw1}.}
    \label{fig:sw1}
\end{figure}

\begin{proof}
We consider three cases depending on the ranks of the vertices $u, w, u_1, w_1$. 
\begin{itemize}
\item If $R_T(u_1) \leq R_T(w_1) \leq R_T(u) = R_T(w)$, then
$$ R_T(u_0) = 1 + R_T(u_1)=R_{T^*}(u_0) $$
and 
$$ R_T(w_0) = 1 + R_T(w_1) \leq 1 + R_T(w) = R_{T^*}(w_0) .$$

\item If $R_T(u_1) \leq R_T(u) = R_T(w) \leq R_T(w_1) $, then
$$ R_T(u_0) = 1 + R_T(u_1)=R_{T^*}(u_0) $$
and 
$$ R_T(w_0) =  1 + R_T(w) = R_{T^*}(w_0) .$$

\item If $R_T(u) = R_T(w) \leq R_T(u_1) \leq R_T(w_1)$, then
$$ R_T(u_0) = 1 + R_T(u) \leq 1 + R_T(u_1) =R_{T^*}(u_0) $$
and 
$$ R_T(w_0) =  1 + R_T(w) = R_{T^*}(w_0) .$$
\end{itemize}

In each case, it is easy to see that the ranks of all other vertices either stay the same or increase from $T$ to $T^*$. Thus, $R(T^*) \geq R(T)$.
\end{proof}


\begin{lemma}\label{lem:sw2}
Let $T$ be a proper binary tree with saturated vertices $u$ and $w$ such that $R_T(u) = R_T(w)$. Let $u$ and $w$'s respective siblings be $u_1$, $w_1$, parents be $u_0$, $w_0$, and assume that $u_0$ is an ancestor of $w_0$. If $R_T(w_1) \geq R_T(u)=R_T(w)$ and $T^* = T - uu_0 - w_1w_0 + uw_0 + w_1u_0$ (see Figure~\ref{fig:sw2}), then $R(T^*) \geq R(T)$.
\end{lemma}
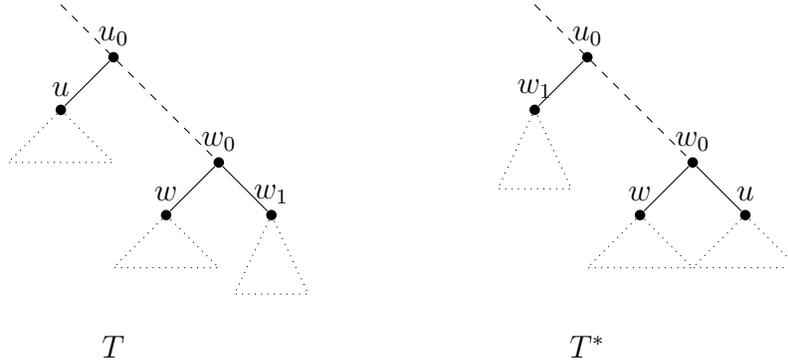
\begin{figure}[H]
\centering
    \begin{tikzpicture}[scale=.7]
        
        \node[fill=black,circle,inner sep=1.4pt] (t2) at (2,2) {};
        \node[fill=black,circle,inner sep=1.4pt] (t4) at (3,3) {};

        \node[fill=black,circle,inner sep=1.4pt] (t1) at (5,1) {};
        \node[fill=black,circle,inner sep=1.4pt] (t8) at (4,0) {};
        \node[fill=black,circle,inner sep=1.4pt] (t9) at (6,0) {};

        \draw [dashed] (2,4)--(t1);
        \draw (t2)--(t4);
        \draw [dotted] (t2)--(1,1)--(3,1)--(t2);

        \draw (t8)--(t1)--(t9);
        \draw [dotted] (t8)--(3,-1)--(5,-1)--(t8);
        \draw [dotted] (t9)--(5.3,-1.5)--(6.7,-1.5)--(t9);

        \node at (3,3.4) {$u_0$};
        \node at (2,2.4) {$u$};

        \node at (5,1.4) {$w_0$};
        \node at (4,.4) {$w$};
        \node at (6,.4) {$w_1$};
        
        \node at (3, -2.5) {$T$};
        
        \begin{scope}[shift={(+9,0)}]
            \node[fill=black,circle,inner sep=1.4pt] (t2) at (2,2) {};
            \node[fill=black,circle,inner sep=1.4pt] (t4) at (3,3) {};

            \node[fill=black,circle,inner sep=1.4pt] (t1) at (5,1) {};
            \node[fill=black,circle,inner sep=1.4pt] (t8) at (4,0) {};
            \node[fill=black,circle,inner sep=1.4pt] (t9) at (6,0) {};

            \draw [dashed] (2,4)--(t1);
            \draw (t2)--(t4);
            \draw [dotted] (t2)--(1.3,.5)--(2.7,.5)--(t2);

            \draw (t8)--(t1)--(t9);
            \draw [dotted] (t8)--(3,-1)--(5,-1)--(t8);
            \draw [dotted] (t9)--(5,-1)--(7,-1)--(t9);

            \node at (3,3.4) {$u_0$};
            \node at (2,2.4) {$w_1$};

            \node at (5,1.4) {$w_0$};
            \node at (4,.4) {$w$};
            \node at (6,.4) {$u$};
            
            \node at (3, -2.5) {$T^*$};
        \end{scope}
        \end{tikzpicture}
    \caption{The trees $T$ (on the left) and $T^*$ (on the right) from Lemma \ref{lem:sw2}.}\label{fig:sw2}
\end{figure}
\begin{proof}
As was done in the proof of Lemma~\ref{lem:sw1}, note that
$$ R_T(w_0) =  1 + R_T(w) = R_{T^*}(w_0)\,, $$
and since $R_T(w_1) \geq R_T(u)$ and $R_{T^*}(u_1)=R_T(u_1)$, we obtain
\begin{align*}
R_{T^*}(u_0) & = \min\{1 + R_{T}(w_1), 1 + R_{T^*}(u_1)\} \\
& \geq \min\{1+R_T(u), 1 + R_{T^*}(u_1)\} = \min\{1+R_T(u), 1 + R_T(u_1)\} = R_T(u_0)\,.
\end{align*}
Since the ranks of all other vertices stay the same or increase from $T$ to $T^*$, we conclude that $R(T^*) \geq R(T)$.
\end{proof}


\begin{lemma}\label{lem:sw3}
Let $T$ be a proper binary tree with saturated vertices $u$ and $w$ such that $R_T(u) = R_T(w)$. Let $u$ and $w$'s respective siblings be $u_1$, $w_1$, parents be $u_0$, $w_0$, and assume that $u_0$ is an ancestor of $w_0$. Suppose that $R_T(w_1) \leq R_T(u)-1 = R_T(w)-1$. Further suppose that either $u_0$ is the root, or its parent $v_1$ satisfies $R_T(v_1) \leq 2 + R_T(w_1)$. Let $T^* = T - uu_0 - w_1w_0 + uw_0 + w_1u_0$  (see Figures~\ref{fig:sw2}~and~\ref{fig:sw3}).
Then we have $R(T^*) \geq R(T).$  
\end{lemma}
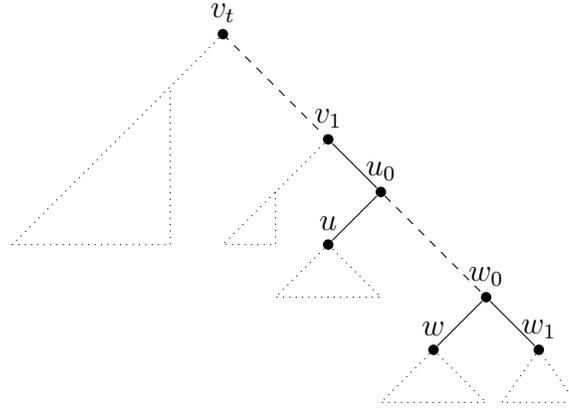
\begin{figure}[H]
\centering
    \begin{tikzpicture}[scale=.7]
        
        \node[fill=black,circle,inner sep=1.4pt] (t2) at (2,2) {};
        \node[fill=black,circle,inner sep=1.4pt] (t4) at (3,3) {};

        \node[fill=black,circle,inner sep=1.4pt] (t1) at (5,1) {};
        \node[fill=black,circle,inner sep=1.4pt] (t8) at (4,0) {};
        \node[fill=black,circle,inner sep=1.4pt] (t9) at (6,0) {};

        \node[fill=black,circle,inner sep=1.4pt] (v1) at (2,4) {};
        \node[fill=black,circle,inner sep=1.4pt] (vt) at (0,6) {};

        \draw [dashed] (0,6)--(v1);
        \draw (v1)--(t4);
        \draw [dashed] (t4)--(t1);
        
        \draw (t2)--(t4);
        \draw [dotted] (t2)--(1,1)--(3,1)--(t2);

        \draw (t8)--(t1)--(t9);
        \draw [dotted] (t8)--(3,-1)--(5,-1)--(t8);
        \draw [dotted] (t9)--(5.3,-1)--(6.7,-1)--(t9);

        \draw [dotted] (v1)--(0,2)--(1,2)--(1,3);
        \draw [dotted] (vt)--(-4,2)--(-1,2)--(-1,5);

        \node at (3,3.4) {$u_0$};
        \node at (2,2.4) {$u$};
        
        \node at (2,4.4) {$v_1$};
        \node at (0,6.4) {$v_t$};
        
        \node at (5,1.4) {$w_0$};
        \node at (4,.4) {$w$};
        \node at (6,.4) {$w_1$};

        \end{tikzpicture}
\caption{The tree $T$ with root $v_t$, with $u_0$ an ancestor of $w_0$ from Lemma~\ref{lem:sw3}.}\label{fig:sw3}
\end{figure}

Note that a difference between this case and that of Lemma~\ref{lem:sw2} is that the rank of $u_0$ may decrease from $T$ to $T^*$, which is compensated for by increasing the rank of $w_0$. 

\begin{proof}
Assume $u_0$ is not the root of $T$. We prove the lemma through a series of claims.

\texttt{Claim 1:} If $R_{T^*}(v_1) < R_T (v_1)$, then $R_{T^*}(u_0)< R_T (u_0)$.

\texttt{Proof 1:} Let $v'_{1}$ be the sibling of $u_0$. By
$R_T(v_1)= \min\{1+R_T(v'_1), 1+R_T(u_0)\}$ and $R_{T^*}(v_1)= \min \{1+R_T(v'_1), 1+R_{T^*}(u_0)\}$,
the assumption $R_{T^*} (v_1) < R_T (v_1)$ implies that $R_{T^*}(u_0) < R_T (u_0)$.

\texttt{Claim 2:} $R_T(u_1) \leq R_{T^*}(u_1)$.

\texttt{Proof 2:} Since $T^{\ast}$ is obtained from $T$ by ``exchanging'' the subtrees $T(u)$ and $T(w_1)$, the assumptions $R_T(u)=R_T(w)$ and $R_T(w_1)< R_T(u)$ imply that $R_T(u_1) \leq R_{T^*}(u_1)$.

\texttt{Claim 3:} If $R_{T^*}(u_0) < R_T (u_0)$, then $R_{T^*}(u_0) \neq 1+R_{T^*}(u_1)$. In particular, $R_{T^*}(u_0) =1+R_T(w_1)$.

\texttt{Proof 3:} We have $R_T(u_0) = \min\{1+R_T(u), 1 + R_T(u_1)\}$, $R_{T^*}(u_0)= \min \{1+R_T(w_1), 1 + R_{T^*}(u_1)\}$.
Assume $R_T^*(u_0) < R_T (u_0)$ and suppose (for contradiction) that $R_{T^*}(u_0) = 1+R_{T^*}(u_1)$. Then $R_{T^*}(u_1) \leq R_T(w_1)$. Using Claim~2 together with our assumption $R_T(w_1) < R_T(u)$, we deduce that $R_T(u_1) \leq R_{T^*}(u_1) < R_T(u)$. Therefore, $R_{T^*} (u_0) < R_T (u_0)$ implies that $R_{T^*}(u_1) < \min \{R_T(u), R_T(u_1) \} = R_T(u_1) \leq R_{T^*}(u_1)$, which is a contradiction. Hence $R_{T^*}(u_0) \neq 1+R_{T^*}(u_1)$. In particular, $R_{T^*}(u_0) =1+R_T(w_1)$.

\texttt{Claim 4:} Let $v_1'$ be the sibling of $u_0$. If $R_{T^*} (v_1) < R_T (v_1)$, then $R_T (v_1') \geq 2+ R_T (w_1)$. and $R_T (u_0) \geq 2+ R_T (w_1)$.

\texttt{Proof 4:} Assume $R_{T^*} (v_1) < R_T (v_1)$ and suppose to the contrary that $R_T (v_1') \leq 1+ R_T (w_1)$. Since $R_T(w_1) < R_T(u)$, we deduce that $R_T (v_1') < 1 + R_T (u)$. Thus, the minimum  distance in $T$ from $v_1$ to a leaf descendant cannot be attained by a vertex of $T(u)$, while the minimum  distance in $T^*$ from $v_1$ to a leaf descendant can be attained at a vertex not belonging to $T(w_1)$. It follows from $R_T(w_1) < R_T(u)=R_T(w)$ and the constructive definition of $T^*$ that $R_T(v_1) \leq R_T^* (v_1)$, a contradiction. The second part of the claim can be proved in a similar way, since $v'_1$ and $u_0$ are siblings. 

\texttt{Claim 5:} If $R_T(v_1) \leq 2+ R_T(w_1)$, then $R_{T^*}(v_1) \geq R_T(v_1)$.

\texttt{Proof 5:} Assume $R_T (v_1) \leq 2 + R_T (w_1)$ and suppose to the contrary that $R_{T^*}(v_1) < R_T (v_1)$. Claim~4 implies both $R_T (v_1') \geq 2+ R_T (w_1)$ and $R_T (u_0) \geq 2+ R_T (w_1)$, which in turn, implies that $R_T (v_1) \geq 3+ R_T (w_1)$, a contradiction.

\medskip
We continue with the proof of the lemma. By Claim~5, $R_{T^*}(v_1) \geq R_T(v_1)$ and this also implies that $R_{T^*}(v_i) \geq R_T(v_i)$ for all $i\in \{1,2,\ldots,t\}$. Note that
\begin{align*}
    R_T(w_0)&=1+R_T(w_1),\\
    R_T(u_0)&=\min\{1+R_T(u),1+R_T(u_1)\},
\end{align*}
and
\begin{align*}
    R_{T^*}(w_0)&=1+R_T(u),\\
    R_{T^*}(u_0)&=\min\{1+R_T(w_1),1+R_{T^*}(u_1)\}.
\end{align*}
We can see that the rank of $w_0$ increases from $T$ to $ T^*$ with an amount of $R_T(u)-R_T(w_1)$. On the other hand, the rank of $u_0$ may decrease with an amount of $\min\{1+R_T(u),1+R_T(u_1)\}-\min\{1+R_T(w_1),1+R_{T^*}(u_1)\}$. However, following the same argument as in the previous lemmata, we have $1+R_T(u_1) \leq 1+R_{T^*}(u_1)$. Since $R_T(w_1) < R_T(u)$, this implies that
\begin{align*}
 & \min\{1+R_T(u),1+R_T(u_1)\}-\min\{1+R_T(w_1),1+R_{T^*}(u_1)\}\\
 &\qquad \leq\min\{1+R_T(u),1+R_{T^*}(u_1)\}-\min\{1+R_T(w_1),1+R_{T^*}(u_1)\}\\
 &\qquad \leq R_T(u)-R_T(w_1).
\end{align*}

Hence, the decrease of the rank of $u_0$ from $T$ to $T^*$ is compensated by the increase of the rank of $w_0$. 
\end{proof}

\medskip
We denote by $u_0 v_1 v_2 \ldots v_t$ the path from $u_0$ to the root, provided $u_0$ is not the root. With the assumptions in Lemma~\ref{lem:sw3}, $R_T(v_1) \leq 2 + R_T(w_1)$ implies that the rank of each of the vertices $v_1, v_2, \cdots, v_t$ either increases or stays the same from $T$ to $T^*$ (if such vertices exist, i.e. if $u_0$ is not the root of $T$). It remains to rule out the case where $R_T(v_1)> 2+ R_T(w_1)$. This will be done as a special case in the next lemma.


\begin{lemma}\label{lem:sw4}
Let $T$ be a proper binary tree with saturated vertices $u$ and $w$ such that $R_T (u) = R_T (w)$. Let $u$ and $w's$ respective siblings be $u_1$, $w_1$, parents be $u_0$, $w_0$, and assume that $u_0$ is an ancestor of $w_0$. Let $w_p$ be the parent of $w_0$. Suppose that $R_T (w_1) < R_T (u) = R_T (w)$ and that $u_0$ is not the root of $T$, while $u_0v_1v_2 \cdots v_t$, $t \in \mathbb{N}$, is the path in $T$ from $u_0$ to the root $v_t$ of $T$. Also, suppose that $R_T (v_1) > R_T (w_1)$.
\begin{itemize}
\item[(1)] If $R_T (v_i) > R_T (w_1)$ does not hold for all $1 \leq i \leq t$, let $s + 1 \in \{2, 3, \cdots , t\}$ be minimal such that $R_T (v_{s+1}) \leq R_T (w_1)$. Then $R(T ^{\ast}) \geq R(T )$, where $T^{\ast} = T - w_0w - w_pw_0 - v_{s+1}v_s + v_{s+1}w_0 + v_sw_0 + w_pw$ (see Figure~\ref{fig:sw4}).
\item[(2)] If $R_T (v_i) > R_T (w_1)$ for all $1 \leq i \leq t$, then $R(T ^{\ast}) \geq R(T )$,
where $T ^{\ast} = T - w_0w - w_pw_0 + w_0v_t + w_pw$.
\end{itemize}
\end{lemma}
\begin{figure}[H]
\centering
    \begin{tikzpicture}[scale=.8]
        
        \node[fill=black,circle,inner sep=1.4pt] (t2) at (2,2) {};
        \node[fill=black,circle,inner sep=1.4pt] (t4) at (3,3) {};

        \node[fill=black,circle,inner sep=1.4pt] (t1p) at (4.25,1.75) {};
        \node[fill=black,circle,inner sep=1.4pt] (t1) at (5,1) {};
        
        \node[fill=black,circle,inner sep=1.4pt] (v1) at (2,4) {};
        \node[fill=black,circle,inner sep=1.4pt] (vt) at (0,6) {};

        \node[fill=black,circle,inner sep=1.4pt] (vx) at (1.5,4.5) {};
        \node[fill=black,circle,inner sep=1.4pt] (vx') at (0.5,5.5) {};

        \node[fill=black,circle,inner sep=1.4pt] (t8) at (4,0) {};
        \node[fill=black,circle,inner sep=1.4pt] (t9) at (6,0) {};

        \draw [dashed] (0,6)--(t4);
        \draw (t4)--(v1);
        \draw [dashed] (v1)--(t1);
        \draw (t2)--(t4);
        \draw [dotted] (t2)--(1,1)--(3,1)--(t2);
        \draw (t1p)--(t1);

        \draw [dotted] (v1)--(0,2)--(1,2)--(1,3);
        \draw [dotted] (vt)--(-2,4)--(-1,4)--(-1,5);
        \draw (vx)--(vx');

        \draw (t8)--(t1)--(t9);
        \draw [dotted] (t8)--(3,-1)--(5,-1)--(t8);
        \draw [dotted] (t1p)--(2.25,1.75-2)--(3.25,1.75-2)--(3.25,0.75);
        \draw [dotted] (t9)--(5.3,-1)--(6.7,-1)--(t9);
        
        \node at (3.3,3.2) {$u_0$};
        \node at (2,2.4) {$u$};
        
        \node at (2.2,4.2) {$v_1$};
        \node at (0,6.4) {$v_t$};
        
        \node at (1.8,4.7) {$v_s$};
        \node at (0.8,5.8) {$v_{s+1}$};
        
        \node at (4.5,2) {$w_p$};
        \node at (5.2,1.3) {$w_0$};
        \node at (4,.4) {$w$};
        \node at (6,.4) {$w_1$};
        
    \end{tikzpicture}
    \qquad
    \raisebox{2em}{
    \begin{tikzpicture}[scale=.8]
        
        \node[fill=black,circle,inner sep=1.4pt] (t2) at (2,2) {};
        \node[fill=black,circle,inner sep=1.4pt] (t4) at (3,3) {};

        \node[fill=black,circle,inner sep=1.4pt] (t1p) at (4.25,1.75) {};
        \node[fill=black,circle,inner sep=1.4pt] (t1) at (5,1) {};
        
        \node[fill=black,circle,inner sep=1.4pt] (v1) at (2,4) {};
        \node[fill=black,circle,inner sep=1.4pt] (vt) at (0,6) {};

        \node[fill=black,circle,inner sep=1.4pt] (vx) at (1.5,4.5) {};
        \node[fill=black,circle,inner sep=1.4pt] (vx') at (0.5,5.5) {};
        \node[fill=black,circle,inner sep=1.4pt] (w0) at (1,5) {};
        \node[fill=black,circle,inner sep=1.4pt] (w1) at (0,4) {};

        \draw [dashed] (0,6)--(t4);
        \draw (t4)--(v1);
        \draw [dashed] (v1)--(t1);
        \draw (t2)--(t4);
        \draw [dotted] (t2)--(1,1)--(3,1)--(t2);
        \draw (t1p)--(t1);

        \draw [dotted] (t1)--(4,0)--(6,0)--(t1);
        \draw [dotted] (w1)--(-.7,3)--(.7,3)--(w1);

        \draw [dotted] (v1)--(0,2)--(1,2)--(1,3);
        \draw [dotted] (t1p)--(2.25,1.75-2)--(3.25,1.75-2)--(3.25,0.75);
        \draw [dotted] (vt)--(-2,4)--(-1,4)--(-1,5);
        \draw (vx)--(vx');
        \draw (w0)--(w1);
        
        \node at (3.3,3.2) {$u_0$};
        \node at (2,2.4) {$u$};
        
        \node at (2.2,4.2) {$v_1$};
        \node at (0,6.4) {$v_t$};
        
        \node at (1.8,4.7) {$v_s$};
        \node at (0.8,5.8) {$v_{s+1}$};
        \node at (1.3,5.2) {$w_0$};
        
        \node at (4.5,2) {$w_p$};
        \node at (5.2,1.3) {$w$};
        \node at (0,4.4) {$w_1$};
        
    \end{tikzpicture}}
    \caption{The tree $T$ on the left, and the tree it generates, $T^*$, on the right from Lemma \ref{lem:sw4}.}
    \label{fig:sw4}
\end{figure}
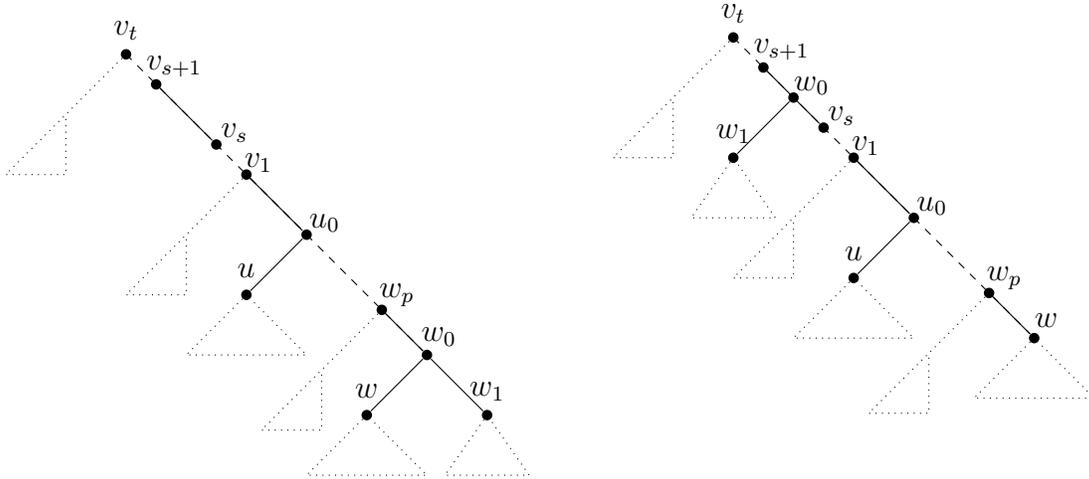
\begin{proof}
Note that the ranks of $w$, $w_1$ and also $w_0$ (since $R_T (v_1) > R_T (w_1)$) stay the same from $T$ to $T^*$. Since $R_T(w_1) \leq R_T(u)-1=R_T(w)-1$, using a similar argument as in the previous lemmata, the ranks of vertices on the path (1) from $v_s$ to $w$ or (2) from $v_t$ to $w$ stay the same or increase from $T$ to $T^*$. Moreover, since in (1) we have that $R_T(v_{s+1}) \leq R_T(w_1)$, the ranks of vertices on the path from $v_t$ to $v_{s+1}$ in this case remain unchanged. Finally, the ranks of all other vertices are unchanged from $T$ to $T^*$.
\end{proof}

\medskip
These lemmata can be summarized as follows:
\begin{itemize}
    \item Lemma~\ref{lem:sw1}: For pairs of saturated vertices of equal rank with neither parent of one vertex being an ancestor of the other vertex, the vertex whose sibling has a lower (or equal) rank can be swapped with the sibling of higher rank without decreasing the security of the tree.
    \item Lemma~\ref{lem:sw2}: For pairs of saturated vertices of equal rank, where the parent of one vertex is an ancestor of the other vertex, if the vertex with the ancestor-parent has lower (or equal) rank than the sibling of the vertex with the descendant-parent, then this sibling and the vertex with the ancestor-parent can be swapped without decreasing the security of the tree.

  \item Lemma~\ref{lem:sw3}: For pairs of saturated vertices of equal rank with one parent being an ancestor of the other, if the vertex with the descendant-parent has strictly higher rank than its sibling, and the neighbor of the ancestor-parent (if it exists) on the path from it to the root has at most two more rank than the sibling of the vertex with the descendant-parent, then this sibling and the vertex with the ancestor-parent can be swapped without decreasing the security of the tree.
  
    \item Lemma~\ref{lem:sw4}: For pairs of saturated vertices of equal rank with one’s parent an ancestor of the other, if the vertex with the descendant-parent has strictly higher rank than its sibling, and this sibling has strictly lower rank than the neighbor of the ancestor-parent (assumed it exists) on the path $P$ from it to the root, then depending on whether this sibling also has strictly lower rank than all other vertices on $P$ or not, this sibling can be moved to just above the ``appropriate'' consecutive vertices of $P$ (see Figure~\ref{fig:sw4}) without decreasing the security of the tree.

\end{itemize}

Note that the four lemmas cover all possible cases. Firstly, given a pair of saturated vertices with equal rank, there is the dichotomy that neither parent of one vertex is an ancestor of the other (Lemma~\ref{lem:sw1}) as opposed to one being an ancestor of the other in the other lemmata. Secondly, the case where the vertex with the ancestor-parent has lower (or equal) rank than the sibling of the vertex with the descendant-parent (Lemma~\ref{lem:sw2}) or when this does not occur (Lemmata~\ref{lem:sw3} and~\ref{lem:sw4}). The final case involves whether or not the ranks of vertices in the path between the ancestor-parent vertex and the root are lower (Lemma~\ref{lem:sw3}) or higher (Lemma~\ref{lem:sw4}) than the relevant sibling vertex. 

 

\subsection{Proof of Theorem~\ref{theo:tl}}

We are now ready to establish the optimality of $T_{\vect{L}}$. To do this, we firstly mention the maximality of $\vect{M}$ with unique entries. 

\begin{claim}
For any proper binary tree $T$ with $\ell$ leaves that corresponds to the partition vector $\vect{M}$, if $\vect{M}$ is not equal to $\vect{L}$, there must be repeated entries in $\vect{M}$. 
\end{claim}

\begin{proof}
It is well known that every positive integer $\ell$ can be uniquely represented as a sum of
distinct binary powers. Therefore, any other representation of the binary powers of $l$ must have repeated entries.
\end{proof}
    

Let $h$ be the minimum $i$ such that $m_i = m_{i+1}$, and let $u$ and $w$ be the corresponding saturated vertices with $|L(T(u))| = 2^{m_h}$ and $|L(T(w))| = 2^{m_{h+1}}$. Depending on the individual ranks of other vertices and the relative locations of $u$ and $w$ in $T$, we may now apply Lemmata~\ref{lem:sw1}, \ref{lem:sw2}, \ref{lem:sw3}, or Lemma~\ref{lem:sw4} and obtain a tree $T^*$ that will satisfy $R(T^*) \geq R(T)$. Two cases may occur:
\begin{itemize}
\item either the original $m_h$ and $m_{h+1}$ get temporarily merged into a $m_h+1$ (in the cases of Lemmata~\ref{lem:sw1}, \ref{lem:sw2}, \ref{lem:sw3}, the vertices $u$ and $w$ become siblings) $-$ the number of saturated vertices decreases in this case; or
\item the distance between $u$ and $w$, if still greater than 2, strictly decreases (in the case of Lemma~\ref{lem:sw4}).
\end{itemize}

Note that every time we apply Lemmata~\ref{lem:sw1}, \ref{lem:sw2}, \ref{lem:sw3}, or Lemma~\ref{lem:sw4}, either the value $h = \min \{i \geq 1\, | \, m_i = m_{i+1} \}$ increases, or $d(u,w)$ decreases while $h$ stays the same. Consequently, this process must terminate after finitely many steps. Therefore, since the binary power representation of any nonnegative integer is unique, by definition, the partition vector $\vect{M}$ (after the required switches) is equal to the binary power representation of $\ell$, $\vect{L}$. 

Let $T_1$ be the resulting tree from performing the ``switches'' to $T$ described in the above paragraph. Since $T_1$ has partition vector $\vect{L}$, every saturated vertex $v$ has a different rank corresponding to the height of $T_1(v)$. Let $u$ be the saturated vertex with the smallest height and $T_1(u)$ the corresponding subtree obtained from $u$ and its descendants. The height of $u$ is then $n_k$ and the number of leaves in $T_1(u)$ is $2^{n_k}$ (recall that $\vect{L}=(n_1, n_2, \ldots, n_k)$, with $n_1>n_2>\cdots>n_k$). The vertex $u$ is the root of $T_1$ if and only if $\ell$ is a binary power. In this case, we are done. Otherwise, let the parent of $u$ be $u_0$, and its sibling be $u_1$. Suppose that $u_0 $ is not the root of $T_1$. Let the parent of $u_0$ be $u_p$ and the root of $T_1$ be $v$ (possibly $v=u_p$). So $u$ is neither the root nor a child of the root.


We consider a new tree $T_1^* = T_1 - u_0u_1 - u_pu_0 + u_pu_1 + u_0v$ (see Figure~\ref{fig:tl1}), rooted at $u_0$.

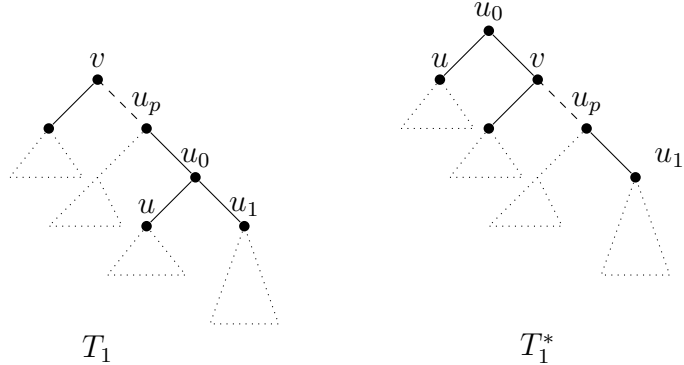
\begin{figure}[H]
\centering
    \begin{tikzpicture}[scale=.65]
        
        \node[fill=black,circle,inner sep=1.4pt] (t2) at (2,2) {};
        \node[fill=black,circle,inner sep=1.4pt] (t4) at (3,3) {};

		  \node[fill=black,circle,inner sep=1.4pt] (tp) at (4,2) {};	              
        \node[fill=black,circle,inner sep=1.4pt] (t1) at (5,1) {};
        \node[fill=black,circle,inner sep=1.4pt] (t8) at (4,0) {};
        \node[fill=black,circle,inner sep=1.4pt] (t9) at (6,0) {};

        \draw [dashed] (3,3)--(tp);
        \draw (tp)--(t1);
        \draw (t2)--(t4);
        \draw [dotted] (t2)--(1.2,1)--(2.7,1)--(t2);

        \draw [dotted] (tp)--(2, 0)--(3.5, 0)--(3, 1);
        \draw (t8)--(t1)--(t9);
        \draw [dotted] (t8)--(3.2,-1)--(4.8,-1)--(t8);
        \draw [dotted] (t9)--(5.3,-2)--(6.7,-2)--(t9);

        \node at (3,3.4) {$v$};

        \node at (5,1.4) {$u_0$};
        \node at (4,.4) {$u$};
        \node at (6,.4) {$u_1$};
        \node at (4,2.5) {$u_p$};

        \node at (3, -2.5) {$T_1$};

        \begin{scope}[shift={(+9,0)}]
            \node[fill=black,circle,inner sep=1.4pt] (t2) at (2,2) {};
            \node[fill=black,circle,inner sep=1.4pt] (t4) at (3,3) {};

			\node[fill=black,circle,inner sep=1.4pt] (tp) at (4,2) {};	              
            \node[fill=black,circle,inner sep=1.4pt] (t1) at (5,1) {};
            \node[fill=black,circle,inner sep=1.4pt] (t1') at (2,4) {};
            \node[fill=black,circle,inner sep=1.4pt] (t8) at (1,3) {};
        
            \draw [dashed] (3,3)--(tp);
            \draw (tp)--(t1);
            \draw (t2)--(t4);
            \draw [dotted] (t2)--(1.2,1)--(2.7,1)--(t2);

            \draw (t8)--(t1')--(t4);
            \draw [dotted] (t8)--(0.2,2)--(1.7,2)--(t8);
            \draw [dotted] (t1)--(4.3,-1)--(5.7,-1)--(t1);
            \draw [dotted] (tp)--(2, 0)--(3.5, 0)--(3, 1);
        
            \node at (3,3.4) {$v$};
            
            \node at (2,4.4) {$u_0$};
            \node at (1,3.4) {$u$};
            \node at (5.7,1.4) {$u_1$};
            \node at (4,2.5) {$u_p$};
            
            \node at (3, -2.5) {$T_1^*$};
        \end{scope}

    \end{tikzpicture}
    \caption{The trees $T_1$ (on the left) and $T_1^*$ (on the right).}
    \label{fig:tl1}
\end{figure}

Following similar arguments as before and making use of the fact that $u$ has the smallest rank among all saturated vertices, one can see that $R(T_1^*) \geq R(T_1)$. In this case, $u$ becomes a child of the root of $T^*$. 

Now consider the case where $u_0$ is the root of $T_1$. As long as $k>2$, we may apply a similar operation to $w$, 
the saturated vertex with $2^{n_{k-1}}$ leaves in $T_1^{*}(w)$, and insert $w_0$ between $u_0$ and $v$ (Figure~\ref{fig:tl2}). For $k = 1$, the sole saturated vertex $u$ is the root, while for $k = 2$, the vertex $w$ is also a child of the root, just like $u$. The resulting tree $T_1^{**}$ has a total rank at least as high as before.

\begin{figure}[H]
\centering
    \begin{tikzpicture}[scale=.7]
        
        \node[fill=black,circle,inner sep=1.4pt] (t2) at (2,2) {};
        \node[fill=black,circle,inner sep=1.4pt] (t4) at (3,3) {};

        \node[fill=black,circle,inner sep=1.4pt] (t1) at (5,1) {};
        \node[fill=black,circle,inner sep=1.4pt] (t1') at (2,4) {};
        \node[fill=black,circle,inner sep=1.4pt] (t8) at (1,3) {};
        
        \node[fill=black,circle,inner sep=1.4pt] (s1') at (1,5) {};
        \node[fill=black,circle,inner sep=1.4pt] (s8) at (0,4) {};   

        \draw [dashed] (3,3)--(t1);
        \draw (t2)--(t4);
        \draw [dotted] (t2)--(1,0)--(3,0)--(t2);

        \draw (t8)--(t1')--(t4);
        \draw (t1')--(s1')--(s8);
        \draw [dotted] (t8)--(0.2,2)--(1.8,2)--(t8);
        \draw [dotted] (s8)--(-.8,3)--(.8,3)--(s8);
        \draw [dotted] (t1)--(4.3,-1)--(5.7,-1)--(t1);

        \node at (3,3.4) {$v$};
        \node at (1,5.4) {$u_0$};
        \node at (0,4.4) {$u$};
        
        \node at (2,4.4) {$w_0$};
        \node at (1,3.4) {$w$};
        \node at (5.3,1.4) {$w_1$};

    \end{tikzpicture}
    \caption{The tree $T_{1}^{**}$.}
    \label{fig:tl2}
\end{figure}
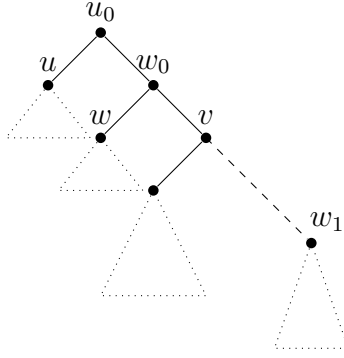

Repeatedly applying the above operation, we may insert the saturated vertices in the ``correct'' order to reach the tree $T_{\vect{L}}$. Since any partition $\vect{M}$ of a proper binary tree $T$ can have the security (weakly) increased to the binary power representation $\vect{L}$, and any tree with partition $\vect{L}$ can have security (weakly) increased to the tree $T_{\vect{L}}$, we conclude that $T_{\vect{L}}$ is indeed a tree which maximizes the security of a tree. \qed

\subsection{The maximum security in a proper binary tree}
As a result of Theorem~\ref{theo:tl}, we can now obtain a lower bound for the maximum security in proper binary trees. 
\begin{theorem}\label{theo:count}
Among proper binary trees $T$ on $\ell\geq 3$ leaves, 
$$ \max R(T) = 2(\ell -  \lfloor \log_2 \ell \rfloor - 1) + z(\ell)\,,$$
where $z(\ell)$ is the number of zeros in the binary representation of $\ell$ (OEIS \href{https://oeis.org/A023416}{A023416}).
\end{theorem}
\begin{proof}
    As a result of Theorem~\ref{theo:tl}, we know that for proper binary trees on $\ell$ leaves, $T_{\vect{L}}$ is a tree which maximizes the security. Thus, given that $\ell = \sum_{i=1}^{k} 2^{n_i}$, the maximum security is simply $R(T_{\vect{L}})$. For each $1\leq i \leq k$,  the saturated vertex $v$ whose $T(v)$ has $2^{n_i}$ leaves and its descendants give us the rank sum
    $$ 1 \times 2^{n_i-1} + 2 \times 2^{n_i-2} + \cdots + (n_i - 1) \times 2 + n_i \times 1 = 2^{n_i+1} -n_i -2 .$$
    Therefore, 
    $$ R(T_{\vect{L}}) = \sum_{i=1}^k (2^{n_i+1} -n_i -2) + \sum_{i=2}^{k} (n_i +1) = 2\ell - n_1 - k - 1 . $$

    Since $n_1 = \lfloor \log_2(\ell)\rfloor$, which is equal to $\lceil \log_2\ell \rceil - 1$ for $\ell$ not a power of $2$ and $\lceil \log_2\ell \rceil$ otherwise, and $k = \lfloor \log_2\ell \rfloor + 1 - z(\ell)$, we can conclude the statement.
\end{proof}


Now that one maximal tree has been fully described, it is interesting to consider other trees that obtain the maximal security, which will be done in Section~\ref{sec:almost-complete}. As an \textit{intermezzo}, we consider some additional swapping lemmata on the maximal trees of Theorem~\ref{theo:tl}. 

\begin{proposition}\label{prop:flips}
   Consider a proper binary tree on $\ell$ leaves $T_{\vect{L}}$, with $\ell$ having binary power representation $\vect{L} = (n_1, n_2, \ldots, n_k)$. If $n_i = n_{i+1}+1$ for fixed $i\in \{2,\dots ,k-1\}$, then using the notation of Definition~\ref{def:T-L}
    with $c_{i}$ the root of the complete binary tree on $2^{n_i}$ leaves, it holds that: 
    \begin{enumerate}
        \item  $R(T_{\vect{L}}) = R(T^*)$, where $T^* = T_{\vect{L}} - v_{i+1}c_{i+1} - v_ic_{i} + v_{i+1}c_{i} + v_ic_{i+1}$, and 
        \item  $R(T_{\vect{L}}) = R(T^*)$, where $T^* = T_{\vect{L}} - v_{i+1}c_{i+1} - v_{i}v_{i-1} + v_{i-1}v_{i+1} + v_ic_{i+1}$.
    \end{enumerate}
\end{proposition}
See Figure~\ref{fig:swaps} for examples of these swaps.

\begin{figure}[ht]
    \raisebox{7em}{(1)}
    \begin{tikzpicture}[scale = 0.7]
        \node (1) at (0, 0) {$\bullet$};
        \node (2) at (-1, -1) {$\bullet$};
        \node (3) at (1, -1) {$\bullet$};
        \node (4) at (0, -2) {$\bullet$};
 
        \draw[-] (1.center) to (2.center);
        \draw[-] (1.center) to (3.center);
        \draw[-] (2.center) to (4.center);
 
        \node at (-0.4, 0.4) {$v_{i+1}$};
        \node at (-1.4, -0.6) {$v_{i}$};
        \node at (1.25, -0.6) {$c_{i+1}$};
        \node at (0.1, -1.6) {$c_{i}$};
       
        \node at (0.25, 0.25) {$\cdot$};
        \node at (0.5, 0.5) {$\cdot$};
        \node at (0.75, 0.75) {$\cdot$};
 
        \node at (-1.25, -1.25) {$\cdot$};
        \node at (-1.5, -1.5) {$\cdot$};
        \node at (-1.75, -1.75) {$\cdot$};
 
        \node (i1) at (1, -2) {};
        \node (i2) at (2, -2) {};
 
        \draw[dotted] (3.center)--(i1.center)--(i2.center)--(3.center);
 
        \node (j1) at (-1, -3) {};
        \node (j2) at (1, -3) {};
 
        \draw[dotted] (4.center)--(j1.center)--(j2.center)--(4.center);
       
    \end{tikzpicture}\quad
    \raisebox{3em}{$\leftrightarrow$}\quad
    \begin{tikzpicture}[scale = 0.7]
        \node (1) at (0, 0) {$\bullet$};
        \node (2) at (-1, -1) {$\bullet$};
        \node (3) at (1, -1) {$\bullet$};
        \node (4) at (0, -2) {$\bullet$};
 
        \draw[-] (1.center) to (2.center);
        \draw[-] (1.center) to (4.center);
        \draw[-] (2.center) to (3.center);
 
        \node at (-0.4, 0.4) {$v_{i+1}$};
        \node at (-1.4, -0.6) {$v_{i}$};
        \node at (1.25, -0.6) {$c_{i+1}$};
        \node at (-0.4, -1.6) {$c_{i}$};
       
        \node at (0.25, 0.25) {$\cdot$};
        \node at (0.5, 0.5) {$\cdot$};
        \node at (0.75, 0.75) {$\cdot$};
 
        \node at (-1.25, -1.25) {$\cdot$};
        \node at (-1.5, -1.5) {$\cdot$};
        \node at (-1.75, -1.75) {$\cdot$};
 
        \node (i1) at (1, -2) {};
        \node (i2) at (2, -2) {};
 
        \draw[dotted] (3.center)--(i1.center)--(i2.center)--(3.center);
 
        \node (j1) at (-1, -3) {};
        \node (j2) at (1, -3) {};
 
        \draw[dotted] (4.center)--(j1.center)--(j2.center)--(4.center);
       
    \end{tikzpicture}\quad
    \raisebox{3em}{$\leftrightarrow$}\quad
    \begin{tikzpicture}[scale = 0.7]
        \node (1) at (0, 0) {$\bullet$};
        \node (2) at (-1, -1) {$\bullet$};
        \node (3) at (1, -1) {$\bullet$};
        \node (4) at (0, -2) {$\bullet$};
 
        \draw[-] (1.center) to (2.center);
        \draw[-] (1.center) to (3.center);
        \draw[-] (2.center) to (4.center);
 
        \node at (-0.4, 0.4) {$v_{i+1}$};
        \node at (-1.4, -0.6) {$v_{i}$};
        \node at (1.25, -0.6) {$c_{i}$};
        \node at (0.8, -2.1) {$c_{i+1}$};
       
        \node at (0.25, 0.25) {$\cdot$};
        \node at (0.5, 0.5) {$\cdot$};
        \node at (0.75, 0.75) {$\cdot$};
 
        \node at (-1.25, -1.25) {$\cdot$};
        \node at (-1.5, -1.5) {$\cdot$};
        \node at (-1.75, -1.75) {$\cdot$};
 
        \node (i1) at (0.2, -1.8) {};
        \node (i2) at (1.8, -1.8) {};
 
        \draw[dotted] (3.center)--(i1.center)--(i2.center)--(3.center);
 
        \node (j1) at (0, -3) {};
        \node (j2) at (1, -3) {};
 
        \draw[dotted] (4.center)--(j1.center)--(j2.center)--(4.center);
       
    \end{tikzpicture}
    
    \raisebox{7em}{(2)}
    \begin{tikzpicture}[scale = 0.7]
        \node (1) at (0, 0) {$\bullet$};
        \node (2) at (-1, -1) {$\bullet$};
        \node (3) at (1, -1) {$\bullet$};
        \node (4) at (0, -2) {$\bullet$};
        \node (5) at (-2, -2) {$\bullet$};
 
        \draw[-] (1.center) to (2.center);
        \draw[-] (1.center) to (3.center);
        \draw[-] (2.center) to (4.center);
        \draw[-] (2.center) to (5.center);
 
        \node at (-0.4, 0.4) {$v_{i+1}$};
        \node at (-1.4, -0.6) {$v_{i}$};
        \node at (-2.4, -1.6) {$v_{i-1}$};
        \node at (1.25, -0.6) {$c_{i+1}$};
        \node at (0.1, -1.6) {$c_{i}$};
       
        \node at (0.25, 0.25) {$\cdot$};
        \node at (0.5, 0.5) {$\cdot$};
        \node at (0.75, 0.75) {$\cdot$};
 
        \node (i1) at (1, -2) {};
        \node (i2) at (2, -2) {};
 
        \draw[dotted] (3.center)--(i1.center)--(i2.center)--(3.center);
 
        \node (j1) at (-0.9, -3) {};
        \node (j2) at (1, -3) {};
 
        \draw[dotted] (4.center)--(j1.center)--(j2.center)--(4.center);
       
        \node (k1) at (-3, -3) {};
        \node (k2) at (-1.1, -3) {};
 
        \draw[dotted] (5.center)--(k1.center)--(k2.center)--(5.center);
       
    \end{tikzpicture}\quad
    \raisebox{3em}{$\leftrightarrow$}\quad
        \begin{tikzpicture}[scale = 0.7]
        \node (1) at (0, 0) {$\bullet$};
        \node (2) at (-1, -1) {$\bullet$};
        \node (3) at (1, -1) {$\bullet$};
        \node (4) at (0, -2) {$\bullet$};
        \node (5) at (-2, -2) {$\bullet$};
 
        \draw[-] (1.center) to (2.center);
        \draw[-] (2.center) to (4.center);
        \draw[-] (2.center) to (3.center);
        \draw[bend right] (1.center) to (5.center);
 
        \node at (-0.4, 0.4) {$v_{i+1}$};
        \node at (-1, -1.4) {$v_{i}$};
        \node at (-2.4, -1.6) {$v_{i-1}$};
        \node at (1.25, -0.6) {$c_{i+1}$};
        \node at (0.1, -1.6) {$c_{i}$};
       
        \node at (0.25, 0.25) {$\cdot$};
        \node at (0.5, 0.5) {$\cdot$};
        \node at (0.75, 0.75) {$\cdot$};
 
        \node (i1) at (1, -2) {};
        \node (i2) at (2, -2) {};
 
        \draw[dotted] (3.center)--(i1.center)--(i2.center)--(3.center);
 
        \node (j1) at (-0.9, -3) {};
        \node (j2) at (1, -3) {};
 
        \draw[dotted] (4.center)--(j1.center)--(j2.center)--(4.center);
       
        \node (k1) at (-3, -3) {};
        \node (k2) at (-1.1, -3) {};
 
        \draw[dotted] (5.center)--(k1.center)--(k2.center)--(5.center);
       
    \end{tikzpicture}\quad
    \raisebox{3em}{$\leftrightarrow$}\quad
    \begin{tikzpicture}[scale = 0.7]
        \node (1) at (0, 0) {$\bullet$};
        \node (2) at (-1, -1) {$\bullet$};
        \node (3) at (1, -1) {$\bullet$};
        \node (4) at (0, -2) {$\bullet$};
        \node (5) at (2, -2) {$\bullet$};
 
        \draw[-] (1.center) to (2.center);
        \draw[-] (1.center) to (3.center);
        \draw[-] (3.center) to (4.center);
        \draw[-] (3.center) to (5.center);
 
        \node at (-0.4, 0.4) {$v_{i+1}$};
        \node at (-1.4, -0.6) {$v_{i-1}$};
        \node at (1.25, -0.6) {$v_{i}$};
        \node at (2.25, -1.6) {$c_{i+1}$};
        \node at (-0.2, -1.6) {$c_{i}$};
       
        \node at (0.25, 0.25) {$\cdot$};
        \node at (0.5, 0.5) {$\cdot$};
        \node at (0.75, 0.75) {$\cdot$};
 
        \node (i1) at (3, -3) {};
        \node (i2) at (2, -3) {};
 
        \draw[dotted] (5.center)--(i1.center)--(i2.center)--(5.center);
 
        \node (j1) at (-1, -3) {};
        \node (j2) at (0.9, -3) {};
 
        \draw[dotted] (4.center)--(j1.center)--(j2.center)--(4.center);
       
        \node (k1) at (-.2, -2) {};
        \node (k2) at (-1.8, -2) {};
 
        \draw[dotted] (2.center)--(k1.center)--(k2.center)--(2.center);
       
    \end{tikzpicture}
   
    \caption{The different swapping configurations described above.}
    \label{fig:swaps}
\end{figure}
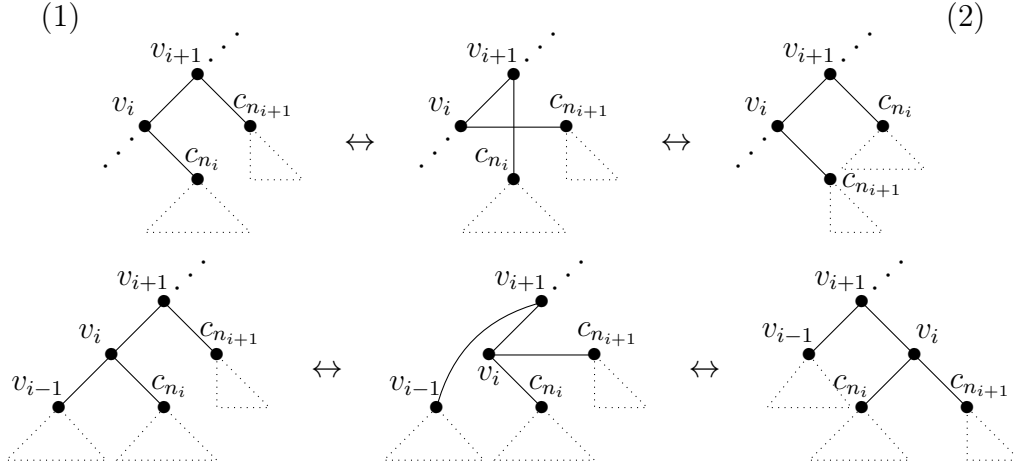

\begin{proof}
Let $T$ be the maximal tree $T_{\vect{L}}$, with $n_1 > n_2 > \cdots > n_k$. Therefore, by construction, $R_T(v_i) = R_T(c_{i})+1$ for $i \geq 2$. Note additionally that since $n_i = n_{i+1}+1$, it holds that $R_T(c_{i}) = R_T(c_{i+1}) + 1$.

Now, consider case (1). As mentioned, $R_T(v_{i+1}) = R_T(c_{i+1})+1$ and $R_T(v_{i}) = R_T(c_{i})+1$. After performing the swap, we have that $R_{T^*}(v_i) = R_T(c_{i+1})+1$, and 
\begin{equation*}
    R_{T^*}(v_{i+1}) = \min\{R_{T^*}(v_i), R_{T^*}(c_{i})\}+1 = \min\{R_T(c_{i+1})+1, R_T(c_{i})\}+1 = R_T(c_{i}) +1.
\end{equation*}
Therefore, this operation does not change the sum of ranks  over all vertices (security). 

Consider case (2). Note that $R_T(v_{i-1}) \geq R_T(c_{i}) = R_T(c_{i+1})+1$. Now, 
\begin{equation*}
    R_{T^*}(v_i) = \min\{R_T(c_{i}), R_T(c_{i+1})\}+1 = R_T(c_{i+1})+1, 
\end{equation*}
and since $R_{T^*}(v_i) \leq R_{T}(v_{i-1}) = R_{T^*}(v_{i-1})$, we get
\begin{equation*}
    R_{T^*}(v_{i+1}) = \min\{R_{T^*}(v_i), R_{T^*}(v_{i-1})\}+1 = R_{T^*}(v_i)+1 = R_T(c_{i+1})+2 = R_T(c_{i})+1.
\end{equation*}
Therefore, this operation also does not change the sum of ranks over all vertices (security). 
\end{proof}

\begin{remark}
It is easy to see that Proposition~\ref{prop:flips} provides a general construction of another family of trees with maximum security.    
\end{remark}

\section{Most protected proper binary trees: Almost complete trees}\label{sec:almost-complete}

We now prove a property of the maximum security trees discussed in the previous section, with the aim of showing that another family of trees has the same property (and thus also achieves maximum security).

\begin{definition}
    Let $\ell$ be a positive integer. Let $F(\ell)$ be a proper binary tree constructed as follows:
    \begin{itemize} 
        \item Create a complete binary tree with $2^{\lfloor \log_2\ell\rfloor}$ leaves, denoted (from left to right) $z_1$, $z_2$, \ldots, $z_{2^{\lfloor \log_2\ell\rfloor}}$.
        \item For $1 \leq i \leq \ell - 2^{\lfloor \log_2\ell\rfloor}$, attach two children to each leaf vertex $z_i$. We call each of these newly attached children \emph{new leaves}. 
    \end{itemize}
\end{definition}
Note that if $\ell$ is a power of two, this process just generates a complete binary tree. If $\ell$ is not a power of two, then some, but not all, leaves of 
the initial complete binary tree have children. An example of a tree $F(\ell)$ is given in Figure~\ref{fig:T_F}.

\begin{figure}[htbp]
    \centering

    \begin{tikzpicture}[scale = 0.9]
        \node (0) at (0, 0){$\bullet$};
        \node (1) at (-2, -1){$\bullet$};
        \node (2) at (2, -1){$\bullet$};
        \node (3) at (-3, -2){$\bullet$};
        \node (4) at (-1, -2){$\bullet$};
        \node (5) at (1, -2){$\bullet$};
        \node (6) at (3, -2){$\bullet$};
        \node (7) at (-3.5, -3){$\bullet$};
        \node (8) at (-2.5, -3){$\bullet$};
        \node (9) at (-1.5, -3){$\bullet$};
        \node (10) at (-0.5, -3){$\bullet$};
        \node (11) at (0.5, -3){$\bullet$};
        \node (12) at (1.5, -3){$\bullet$};

        \draw[-] (0.center) to (1.center);
        \draw[-] (0.center) to (2.center);
        \draw[-] (1.center) to (3.center);
        \draw[-] (1.center) to (4.center);
        \draw[-] (3.center) to (7.center);
        \draw[-] (3.center) to (8.center);
        \draw[-] (4.center) to (9.center);
        \draw[-] (4.center) to (10.center);
        \draw[-] (2.center) to (5.center);
        \draw[-] (2.center) to (6.center);
        \draw[-] (5.center) to (11.center);
        \draw[-] (5.center) to (12.center);
    \end{tikzpicture}
    \qquad 
    \begin{tikzpicture}[scale = 0.9]
        \node (0) at (-1, 0){$\bullet$};
        \node (1) at (-2, -1){$\bullet$};
        \node (2) at (0, -1){$\bullet$};
        \node (3) at (-3, -2){$\bullet$};
        \node (4) at (-0.5, -2){$\bullet$};
        \node (5) at (-4, -3){$\bullet$};
        \node (6) at (-2, -3){$\bullet$};
        \node (7) at (-1, -3){$\bullet$};
        \node (8) at (0, -3){$\bullet$};
        \node (9) at (-4.5, -4){$\bullet$};
        \node (10) at (-3.5, -4){$\bullet$};
        \node (11) at (-2.5, -4){$\bullet$};
        \node (12) at (-1.5, -4){$\bullet$};

        \draw[-] (0.center) to (1.center);
        \draw[-] (0.center) to (2.center);
        \draw[-] (1.center) to (3.center);
        \draw[-] (1.center) to (4.center);
        \draw[-] (3.center) to (5.center);
        \draw[-] (3.center) to (6.center);
        \draw[-] (4.center) to (7.center);
        \draw[-] (4.center) to (8.center);
        \draw[-] (5.center) to (9.center);
        \draw[-] (5.center) to (10.center);
        \draw[-] (6.center) to (11.center);
        \draw[-] (6.center) to (12.center);
    \end{tikzpicture}
    \caption{The  tree $F(\ell)$ corresponding to $\ell = 7$ (left) and the $T_{\vect{L}}$ tree corresponding to the same $\ell$ (right). Note 
    that the security, 8, is the same in both.}
    \label{fig:T_F}
\end{figure}
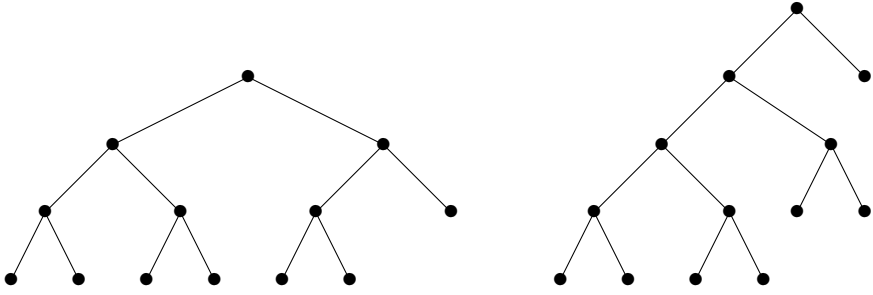

These ``almost complete trees'' were first introduced as ``good trees'' in \cite{Szekely-Wang:2005}, and have been ever since shown to be an extremal structure with respect to many different graph invariants. 

\medskip
We give an alternative proof that $R(F(\ell)) = R(T_{\vect{L}})$, which shows constructively that the sums of ranks are equal. In addition, this gives an algorithm for the construction of an $F(\ell)$ tree from the binary power representation of $\ell$.

\medskip
\begin{description}
\item[Algorithm for constructing $F(\ell)$]\phantom{-}

Let $\ell=\sum_{i=1}^k2^{n_i}$ with $n_1> n_2>\dots >n_k\geq 0$. Suppose $n_1 \geq 1$. 

Let $d_i=n_{i-1}-n_{i}$ for each $1<i\leq k$.

For any $m\in \N$, let $c_m$ be a complete binary tree with $2^m$ leaves.

If $q_i$ is the root of a $c_{n_i}$ tree, let $\ell_i$ and $r_i$ be the left and right child of $q_i$, respectively (if they exist).

\begin{itemize}
\item \textbf{Step 1}: Construct the tree $c_{n_1}$ with root $q_1$. This tree has a $c_{n_1-1}$ tree rooted at both $\ell_1$ and $r_1$.   

\item \textbf{Step $i$}: For $2\leq i\leq k$, consider the $c_{n_{i-1}}$ tree rooted at $q_{i-1}$ and reconsider the $c_{n_{i-1}-1}$ tree rooted at $\ell_{i-1}$ as a $c_{d_{i}-1}$ tree with a $c_{n_{i}}$ tree rooted at each of its leaves. There are $2^{d_i-1}$ such $c_{n_i}$ trees, accounting for $2^{d_i-1}2^{n_i} = 2^{n_{i-1}-1}$. If $d_i>1$, consider the leftmost $c_{n_i}$ tree, and call its root $\rho_i$. If $d_i=1$, we have a $c_{n_i}$ tree rooted at $\ell_{i-1}$, so relabel $\ell_{i-1}$ as $\rho_i$. To each of the leaf vertices of the $c_{n_i}$ tree rooted at $\rho_i$, add two leaf children so that we now have a $c_{n_i+1}$ tree rooted at $\rho_i$.

Note that $\rho_i$ has a sibling vertex (possibly $r_{i-1}$, if $\rho_i = \ell_{i-1}$) and label this vertex $q_i$. In both cases, $q_i$ is the root of a $c_{n_i}$ tree. The tree at the end of step $i$ is the tree $F(2^{n_1}+\cdots+2^{n_i})$. 
\end{itemize}
\end{description}
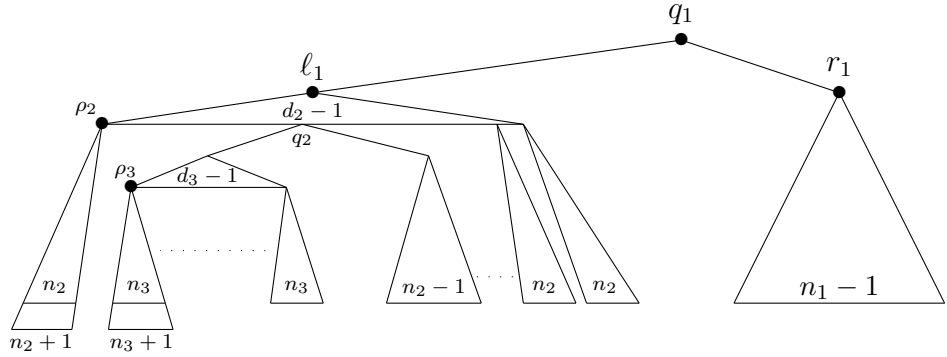
\begin{figure}[htbp]
        \begin{tikzpicture}[scale = 0.7]
            \node (1) at (0, 0) {$\bullet$}; 
            
            \node (2) at (-7, -1) {$\bullet$}; 
            \node (3) at (3, -1) {$\bullet$}; 

            \node at (0, 0.5) {$q_1$};
            \node at (-7, -0.5) {$\ell_1$};
            \node at (3, -0.5) {$r_1$};

            \node (4) at (1,-5) { }; 
            \node (5) at (5,-5) { };
            \node (21) at (3,-4.7) {\small $n_1-1$};
            
            \node (6) at (-3,-1.6) { };
            \node (7) at (-11,-1.6) {$\bullet$};
            \node at (-11.3,-1.3) {\tiny $\rho_2$};
            \node (20) at (-7,-1.35) {\tiny $d_2-1$};
            \node at (-7.2,-1.9) {\tiny $q_2$};

            \node (8) at (-0.8,-5) { };
            \node (9) at (-1.8,-5) { };

            \node (10) at (-2,-5) { };
            \node (11) at (-3,-5) { };
            \node (12) at (-3.5,-1.6) { };


            \node (14) at (-12.5,-5) { };
            \node (15) at (-11.5,-5) { };
            \node (25) at (-11.9,-4.7) {\tiny $n_2$};
            \node (40) at (-12.72,-5.5) { };
            \node (41) at (-11.5735,-5.5) { };
            \draw[-] (14.center) to (40.center);
            \draw[-] (15.center) to (41.center);
            \draw[-] (41.center) to (40.center);
            \node (43) at (-12.16,-5.75) {\tiny $n_2+1$};

            \node (16) at (-9,-2.2) { };
            \node (17) at (-4.8,-2.2) { };
            \node (18) at (-7.2,-1.6) { };

            
            \node (22) at (-1.45,-4.7) {\tiny $n_2$};
            \node (23) at (-2.6,-4.7) {\tiny $n_2$};
            \draw[-,loosely dotted] (-3.2,-4.5) to (-3.9,-4.5);

            \node (25) at (-10.8,-5) { };
            \node (26) at (-9.8,-5) { };
            \node (30) at (-10.3,-4.7) {\tiny $n_3$};
            \node (44) at (-10.8796,-5.5) { };
            \node (45) at (-9.65225,-5.5) { };
            \draw[-] (44.center) to (25.center);
            \draw[-] (26.center) to (45.center);
            \draw[-] (44.center) to (45.center);
            \node at (-10.25,-5.75) {\tiny $n_3+1$};

            \node (27) at (-6.8,-5) { };
            \node (28) at (-7.8,-5) { };
            \node (29) at (-7.3,-4.7) {\tiny $n_3$};

            \node (32) at (-3.8,-5) { };
            \node (33) at (-5.6,-5) { };
            \node (34) at (-4.7,-4.7) {\tiny $n_2-1$};
            
            \draw[-] (17.center) to (32.center);
            \draw[-] (17.center) to (33.center);
            \draw[-] (33.center) to (32.center);

            \node (35) at (-10.45,-2.8) {$\bullet$};
            \node at (-10.55,-2.5) {\tiny $\rho_3$};
            \node (36) at (-7.5,-2.8) { };
            \node (37) at (-9,-2.59) {\tiny $d_3-1$ };

            \draw[-] (16.center) to (35.center);
            \draw[-] (16.center) to (36.center);
            \draw[-] (36.center) to (35.center);

            \draw[-] (35.center) to (25.center);
            \draw[-] (35.center) to (26.center);
            \draw[-] (26.center) to (25.center);

            \draw[-] (36.center) to (27.center);
            \draw[-] (36.center) to (28.center);
            \draw[-] (27.center) to (28.center);

            \draw[-,loosely dotted] (-7.9,-4) to (-9.9,-4);


            \draw[-] (18.center) to (16.center);
            \draw[-] (18.center) to (17.center);

            \draw[-] (7.center) to (14.center);
            \draw[-] (7.center) to (15.center);
            \draw[-] (14.center) to (15.center);

            \draw[-] (3.center) to (4.center);
            \draw[-] (3.center) to (5.center);
            \draw[-] (5.center) to (4.center);

            \draw[-] (2.center) to (6.center);
            \draw[-] (2.center) to (7.center);
            \draw[-] (6.center) to (7.center);

            \draw[-] (6.center) to (8.center);
            \draw[-] (6.center) to (9.center);
            \draw[-] (8.center) to (9.center);

            \draw[-] (12.center) to (10.center);
            \draw[-] (12.center) to (11.center);
            \draw[-] (10.center) to (11.center);

            \draw[-] (1.center) to (2.center);
            \draw[-] (1.center) to (3.center);

        \end{tikzpicture}
        \caption{$F(\ell)$ for $2^{n_1}+2^{n_2}+2^{n_3}$ leaves where $n_1=n_2+d_2$, and $n_2=n_3+d_3$.}
        \label{tee_eff_ell1}
\end{figure}
See Figure~\ref{tee_eff_ell1} for an illustration. Following the above algorithm, we can verify that the total number of leaves added is, in fact, equal to $\ell$.
\begin{align*}
\text{Total leaves in }F(\ell)&=\underbrace{\left(\sum_{i=2}^k2^{n_i+1}\right)}_{\#\text{ of leaves at height }n_1+1}+\underbrace{2^{n_1}-\left(\sum_{i=2}^k2^{n_i}\right)}_{\#\text{ of leaves at height }n_1}=\ell.
\end{align*}

\begin{proposition}
   Suppose $\ell=\sum_{i=1}^k2^{n_i}$ with $n_1> n_2>\dots >n_k\geq 0$. Then 
   $$R(F(\ell))=R(T_{\bm{L}})\,.$$
\end{proposition}

\begin{proof}

Note that the security of a $c_{n_i}$ tree is $$R(c_{n_i})=2^{n_i+1}-n_i-2.$$

To verify the rank formula, let $S_i$ be the tree at the end of step $i$ of the algorithm for constructing $F(\ell)$. Then $S_1$ is a $c_{n_1} = F(2^{n_1})$, $S_2$ is a $F(2^{n_1}+2^{n_2})$ tree, and so on. Then $R(S_i)=R(S_{i-1})+R(c_{n_{i}+1})-R(c_{n_i})$ since at each step one removes a $c_{n_i}$ tree and replaces it with a $c_{n_{i}+1}$ tree. 

Clearly, this operation does not affect the ranks of the remaining vertices. Indeed, since at each step, $i\geq 2$, one splits a $c_{n_{i-1}}$ tree into two or more $c_{n_i}$, the removed root $\rho_i$ always has a sibling. Consequently, $\rho_i$ always has the closest leaf within another $c_{n_i}$ tree (or was originally a leaf). As a result, removing the root $\rho_i$ does not alter the ranks of any other nodes.
Then
\begin{align*}
R(F(\ell))&=R(S_k)\\
&=R(S_{k-1})+R(c_{n_{k}+1})-R(c_{n_k})\\
&=R(c_{n_1})+\sum_{i=2}^k(R(c_{n_i+1})-R(c_{n_i}))\\
&=R(c_{n_1})+\sum_{i=2}^k\big[(2^{n_i+2}-(n_i+1)-2)-(2^{n_i+1}-n_i-2)\big]\\
&=R(c_{n_1})+\sum_{i=2}^k(2^{n_i+1}-1)\\
&=2^{n_1+1}-n_1-2+2(\ell-2^{n_1})-(k-2+1)\\
&=2\ell-n_1-k-1=R(T_{\bm{L}}).
\end{align*}
\end{proof}

\section{The most protected vertex/root}\label{sec:most-protected}

First we point out the trivial, but still interesting, observation that in a tree $T$ rooted at $r$, the maximum rank $R_T(v)$ is not necessarily achieved at the root. This can be seen from the simple example in Figure~\ref{fig:ex1}, where the root $r$ has rank 1 and the vertex $v$ can have an arbitrarily large rank. 

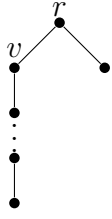
\begin{figure}[htbp]
\centering
    \begin{tikzpicture}[scale=.6]
        \node[fill=black,circle,inner sep=1.4pt] (t1) at (3,0) {};
        \node[fill=black,circle,inner sep=1.4pt] (t2) at (3,1) {};
        \node[fill=black,circle,inner sep=1.4pt] (t4) at (3,2) {};
		\node[fill=black,circle,inner sep=1.4pt] (t6) at (3,3) {};
        \node[fill=black,circle,inner sep=1.4pt] (t8) at (4,4) {};
        \node[fill=black,circle,inner sep=1.4pt] (t9) at (5,3) {};

        \draw (t1)--(t2);
        \draw (t4)--(t6);
		\draw (t6)--(t8);
        \draw (t8)--(t9);
        \node at (3,1.65) {$\vdots$};
        \node at (4,4.3) {$r$};
        \node at (3,3.4) {$v$};

    \end{tikzpicture}
    \caption{A rooted tree with $R_T(r)=1$ and $R_T(v)>1$. }\label{fig:ex1}
\end{figure}

There do not seem to be any obvious characteristics of the most protected vertices in a tree. Let the set of such vertices of $T$ be denoted by $R_{max}(T)$. It may contain only one vertex (as in Figure~\ref{fig:ex1}) or multiple vertices. When $R_{max}(T)$ contains multiple vertices, they may induce a subtree (i.e. they are all adjacent) as in Figure~\ref{fig:bin} (where all internal vertices in $R_{max}(T)$); they may also form an independent set as in Figure~\ref{fig:ex2}.

\begin{figure}[htbp]
\centering
    \begin{tikzpicture}[scale=.6]
        \node[fill=black,circle,inner sep=1.4pt] (t1) at (0,0) {};
        \node[fill=black,circle,inner sep=1.4pt] (t2) at (1,1) {};
        \node[fill=black,circle,inner sep=1.4pt] (t3) at (2,0) {};
        \node[fill=black,circle,inner sep=1.4pt] (t4) at (2,2) {};
        \node[fill=black,circle,inner sep=1.4pt] (t5) at (3,1) {};
				        \node[fill=black,circle,inner sep=1.4pt] (t6) at (3,3) {};
        \node[fill=black,circle,inner sep=1.4pt] (t7) at (4,2) {};
        \node[fill=black,circle,inner sep=1.4pt] (t8) at (4,4) {};
        \node[fill=black,circle,inner sep=1.4pt] (t9) at (5,3) {};

        \draw (t1)--(t2);
		\draw (t4)--(t8);
        \draw (t2)--(t3);
        \draw (t4)--(t5);
        \draw (t7)--(t6);
        \draw (t8)--(t9);
        \node at (1.5,1.7) {$\iddots$};

        \end{tikzpicture}
\caption{A rooted proper binary caterpillar.}\label{fig:bin}
\end{figure}
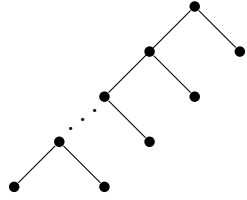

\begin{figure}[htbp]
\centering
    \begin{tikzpicture}[scale=.6]
        \node[fill=black,circle,inner sep=1.4pt] (t1) at (0,0) {};
        \node[fill=black,circle,inner sep=1.4pt] (t2) at (1,1) {};
        \node[fill=black,circle,inner sep=1.4pt] (t3) at (2,0) {};
        \node[fill=black,circle,inner sep=1.4pt] (t4) at (2,2) {};
        \node[fill=black,circle,inner sep=1.4pt] (t5) at (3,1) {};
 \node[fill=black,circle,inner sep=1.4pt] (t5') at (4,0) {};
				        \node[fill=black,circle,inner sep=1.4pt] (t6) at (3,3) {};
        \node[fill=black,circle,inner sep=1.4pt] (t7) at (4,2) {};
        \node[fill=black,circle,inner sep=1.4pt] (t8) at (4,4) {};
        \node[fill=black,circle,inner sep=1.4pt] (t9) at (5,3) {};
\node[fill=black,circle,inner sep=1.4pt] (t9') at (6,2) {};

        \draw (t1)--(t8);
        \draw (t2)--(t3);
        \draw (t4)--(t5');
        \draw (t7)--(t6);
        \draw (t8)--(t9');
       \node at (4,4.3) {$u$};
\node at (2,2.3) {$v$};

        \end{tikzpicture}
\caption{A rooted tree with two most protected vertices $u$ and $v$.}\label{fig:ex2}
\end{figure}
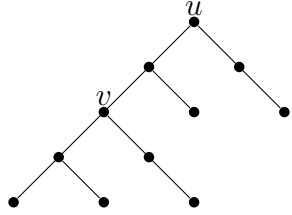

Given the motivation for vertex ranks and ``protectedness'', it is natural to find, among a class of rooted trees, the most protected vertex (among all possible tree structures) and its rank. 
That is, for a given class of rooted trees $\mathcal{T}$ we want to find
$$ R_{\mathcal{T}} := \max_{T \in \mathcal{T}} \max_{v \in V(T)} R_T(v)\,.$$
Let $R_T(r)$ denote the rank of the root $r$ in $T$. We first show a useful fact that among all trees of a given order $R_{\mathcal{T}}$ can be achieved at a root.

\subsection{General trees}

First note the trivial fact that the rank of any vertex in a tree on $n$ vertices cannot be possibly larger than $n-1$.

\begin{proposition}\label{prop:path1}
Let $\mathcal{T}_n$ denote the set of trees with order $n$. Then $R_{\mathcal{T}_n} = R_T(r)$ for some $T \in \mathcal{T}_n$ with root $r$. Moreover, $$ R_T(r) \leq n-1 $$
with equality if and only if $T$ is a path and $r$ is one of its end vertices. 
\end{proposition}

\begin{proof}
To see this, we show that for any $T \in \mathcal{T}_n$ and $v \in V(T)$, there exists a $T' \in \mathcal{T}_n$ with root $r'$ such that
$$ R_{T'}(r') \geq R_T(v) . $$
If $v$ is already the root of $T$, there is nothing to prove. Otherwise, let $r$ be the root of $T$, $u$ be the unique parent of $v$ in $T$, and $w$ be one of the leaf descendants of $v$ in $T$. Consider the tree $T'=T - uv + rw$ as in Figure~\ref{fig:ex3}. It is obvious that $v=r'$ is now the root of $T' \in \mathcal{T}_n$.

\begin{figure}[H]
\centering
    \begin{tikzpicture}[scale=.6]
        \node[fill=black,circle,inner sep=1.4pt] (t1) at (1,0) {};
        \node[fill=black,circle,inner sep=1.4pt] (t2) at (2,1) {};
        \node[fill=black,circle,inner sep=1.4pt] (t3) at (3,0) {};
        \node[fill=black,circle,inner sep=1.4pt] (t4) at (3,2) {};
        \node[fill=black,circle,inner sep=1.4pt] (t5) at (4,1) {};
		\node[fill=black,circle,inner sep=1.4pt] (t6) at (3,3) {};
        \node[fill=black,circle,inner sep=1.4pt] (t8) at (4,4) {};
        \node[fill=black,circle,inner sep=1.4pt] (t9) at (5,3) {};

        \draw (t1)--(t4);
        \draw (t4)--(t6);
        \draw (t6)--(t8);
        \draw (t2)--(t3);
        \draw (t4)--(t5);
        \draw (t8)--(t9);
        \node at (2.6,3.1) {$u$};
        \node at (4,4.3) {$r$};
        \node at (2.6,2.1) {$v$};
        \node at (4.1,1.3) {$w$};

        \node[fill=black,circle,inner sep=1.4pt] (s1) at (0+7,0) {};
        \node[fill=black,circle,inner sep=1.4pt] (s2) at (1+7,1) {};
        \node[fill=black,circle,inner sep=1.4pt] (s3) at (2+7,0) {};
        \node[fill=black,circle,inner sep=1.4pt] (s4) at (2+7,2) {};
        \node[fill=black,circle,inner sep=1.4pt] (s5) at (3+7,1) {};
        \node[fill=black,circle,inner sep=1.4pt] (s7) at (3+7,0) {};
        \node[fill=black,circle,inner sep=1.4pt] (s9') at (4+7,-1) {};
        \node[fill=black,circle,inner sep=1.4pt] (s9'') at (2+7,-1) {};
        
        \draw (s1)--(s4);
        \draw (s2)--(s3);
        \draw (s4)--(s5);
        \draw (s9'')--(s7);
        \draw (s7)--(s9');
        \draw (s7)--(s5);
        \node at (2+7.4,-1) {$u$};
        \node at (3+7.4,0) {$r$};
        \node at (2+7,2.3) {$v$};
        \node at (3+7.2,1.3) {$w$};
        \end{tikzpicture}
\caption{An example of the trees $T$ (on the left), $T'$ (on the right) and the vertices $u,v,w$.}\label{fig:ex3}
\end{figure}
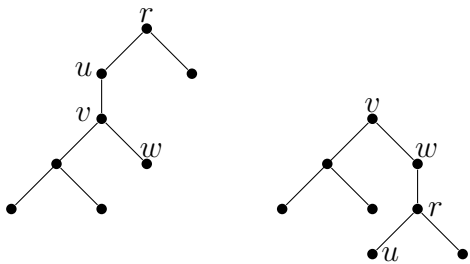

Note that from $T$ to $T'$, the distance between $v$ and any of its original leaf descendants (in $T$, with the exception of $w$, which is now not a leaf) remains the same. The distance between $v$ and any of its new leaf descendants (in $T'$) is greater than the distance between $v$ and $w$. Hence
$$ R_{T'}(v) \geq R_T(v). $$
Since the rank of any vertex in a tree on $n$ vertices cannot be possibly larger than $n-1$, the second part of the proposition follows. 
\end{proof}

\medskip
In the rest of the section we focus on finding the maximum rank of the root in various classes of trees. 

In practice it is often of interest to fix the degree of the root, $k$. In this case the \emph{starlike tree} $S(\ell_1, \ell_2, \ldots , \ell_k)$ is the tree obtained by identifying one end-vertex of each of the paths of length $\ell_1, \ell_2, \ldots , \ell_k$ respectively.

\begin{proposition}\label{prop:starlike1}
For a tree $T$ of order $n>1$, rooted at $r$ with degree $k \geq 3$, 
$$ R_T(r) \leq \left\lfloor \frac{n-1}{k} \right\rfloor $$
with equality if $T$ is a starlike tree $S(\ell_1, \ell_2, \ldots , \ell_k)$, rooted at the only branching vertex (vertex of degree $\geq 3$) with 
$$ \min_{1\leq i \leq k} \{\ell_i \} = \left\lfloor \frac{n-1}{k} \right\rfloor . $$
\end{proposition}

\begin{proof}
Let the neighbourhood of $r$ be $\{ v_1, \ldots , v_k \}$ and let the connected component containing $v_i$ in $T-r$ be $T_i$, rooted at $v_i$. 
Since $n-1=|V(T)|-1 = \sum_{i=1}^k |V(T_i)|$, we have 
$$|V(T_{i_0})|:=\min_{1\leq i \leq k} \{ |V(T_i)| \} \leq \left\lfloor \frac{n-1}{k} \right\rfloor , $$
and by Proposition~\ref{prop:path1}
$$ R_{T_{i_0}}(v_{i_0}) \leq \left\lfloor \frac{n-1}{k} \right\rfloor -1 . $$
The conclusion then follows from the fact that
$$ R_T(r) = 1+ \min_{1\leq i \leq k} \{ R_{T_i}(v_i) \} . $$
\end{proof}

\begin{remark}
The starlike tree in our statement is by no means the unique tree achieving the maximum root rank. However, the $S(\ell_1, \ell_2, \ldots , \ell_k)$ with the condition that $|\ell_i - \ell_j| \leq 1$ for any $1\leq i,j\leq k$, which certainly maximizes the root rank, has been shown to be extremal with respect to many graph invariants and, in particular, to be the densest among trees with a given order and number of leaves. 
\end{remark}

\subsection{\textit{k}-ary trees} 
By $k$-ary trees, we mean rooted trees in which the outdegrees are bounded above by $k$, i.e. every vertex has at most $k$ children. 

Let $\mathcal{T}_n^k$ denote the set of $k$-ary trees of order $n$. First we show that Proposition~\ref{prop:path1} can be generalized to $\mathcal{T}_n^k$.

\begin{proposition}\label{prop:prop1'}
Let $\mathcal{T}_n^k$ denote the set of $k$-ary trees with order $n$. Then $R_{\mathcal{T}_n^k} = R_T(r)$ for some $T \in \mathcal{T}_n^k$ with root $r$.
\end{proposition}

\begin{proof}
Following the same idea as the proof of Proposition~\ref{prop:path1}, let $T\in \mathcal{T}_n^k$ and $v \in V(T)$ be a non-root vertex. Let $\mathcal{C}_v$ denote the set of children of $v$ in $T$, we use $S_1$ to denote the component containing $v$ in $T-\mathcal{C}_v$ and let $S_2 = T - \{ S_1 - v \}$. Let $w$ be one of the leaf descendants of $v$. We now consider the tree $T'$ resulted from identifying $r$ (of $S_1$) with $w$ (of $S_2$). See Figure~\ref{fig:ex3'}.

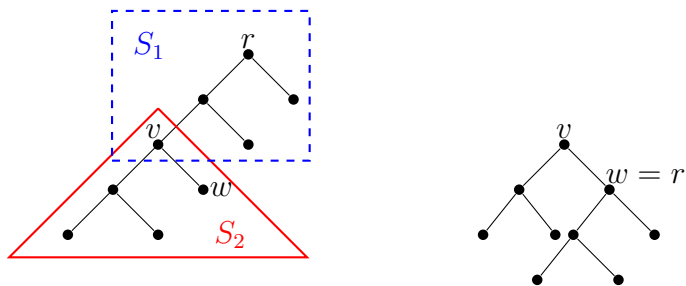
\begin{figure}[htbp]
\centering
    \begin{tikzpicture}[scale=.6]
        \node[fill=black,circle,inner sep=1.4pt] (t1) at (0,0) {};
        \node[fill=black,circle,inner sep=1.4pt] (t2) at (1,1) {};
        \node[fill=black,circle,inner sep=1.4pt] (t3) at (2,0) {};
        \node[fill=black,circle,inner sep=1.4pt] (t4) at (2,2) {};
        \node[fill=black,circle,inner sep=1.4pt] (t5) at (3,1) {};
		\node[fill=black,circle,inner sep=1.4pt] (t6) at (3,3) {};
        \node[fill=black,circle,inner sep=1.4pt] (t8) at (4,4) {};
        \node[fill=black,circle,inner sep=1.4pt] (t9) at (5,3) {};
        \node[fill=black,circle,inner sep=1.4pt] (t9') at (4,2) {};

        \draw (t1)--(t8);
        \draw (t2)--(t3);
        \draw (t4)--(t5);
        \draw (t8)--(t9);
        \draw (t6)--(t9');
        \node at (4,4.3) {$r$};
        \node at (1.9,2.3) {$v$};
        \node at (3.4,1) {$w$};
        \node[blue] (s1) at (1.8, 4.2){$S_1$};
        \node[red] (s2) at (3.6, 0){$S_2$};  

        \draw[-, red, thick] (2, 2.8)--(-1.3, -0.5)--(5.3, -0.5)--(2, 2.8);
        \node[draw = blue, dashed, thick, fit=(t4) (t6) (t8) (t9) (t9') (s1)] {};

        \node[fill=black,circle,inner sep=1.4pt] (s1) at (0.2+9,0) {};
        \node[fill=black,circle,inner sep=1.4pt] (s2) at (1+9,1) {};
        \node[fill=black,circle,inner sep=1.4pt] (s3) at (1.8+9,0) {};
        \node[fill=black,circle,inner sep=1.4pt] (s4) at (2+9,2) {};
        \node[fill=black,circle,inner sep=1.4pt] (s5) at (3+9,1) {};
		\node[fill=black,circle,inner sep=1.4pt] (s6) at (2.2+9,0) {};
        \node[fill=black,circle,inner sep=1.4pt] (s7) at (4+9,0) {};
        \node[fill=black,circle,inner sep=1.4pt] (s9') at (3.2+9,-1) {};
        \node[fill=black,circle,inner sep=1.4pt] (s9'') at (1.4+9,-1) {};       

        \draw (s2)--(s4);
        \draw (s1)--(s2);
        \draw (s2)--(s3);
        \draw (s4)--(s5)--(s7);
        \draw (s5)--(s6);
        \draw (s9'')--(s6);
        \draw (s6)--(s9');
        \node at (3+9.8,1.3) {$w = r$};
        \node at (2+9,2.3) {$v$};

    \end{tikzpicture}
    \caption{The trees $T$ (on the left), $T'$ (on the right) and the vertices $r,v,w$.}\label{fig:ex3'}
\end{figure}

It is easy to verify that $T' \in \mathcal{T}_n^k$ and that 
$$ R_{T'}(v) \geq R_T(v) $$
following the same reasoning as that of Proposition~\ref{prop:path1}.
\end{proof}

Thus, again we only need to consider the root ranks. For our extremal structure, Figure~\ref{fig:com} shows a \emph{rooted complete ternary tree}, defined in general as the following.

\begin{definition}
A rooted complete $k$-ary tree is a proprer $k$-ary tree, all of whose leaves are at two consecutive levels (vertices at the same distance from the root are at the same level) and there is at most one vertex that has both leaf and non-leaf children.
\end{definition}

\begin{figure}[htbp]

\centering
    \begin{tikzpicture}[scale=0.5]
        \node[fill=black,circle,inner sep=1.5pt] (v) at (10,6) {}; 
        \node[fill=black,circle,inner sep=1.5pt] (v1) at (4,4) {};
        \node[fill=black,circle,inner sep=1.5pt] (v2) at (10,4) {};
        \node[fill=black,circle,inner sep=1.5pt] (v4) at (16,4) {};
        \node[fill=black,circle,inner sep=1.5pt] (v11) at (2,2) {};
        \node[fill=black,circle,inner sep=1.5pt] (v12) at (4,2) {};
        \node[fill=black,circle,inner sep=1.5pt] (v13) at (6,2) {};
        \node[fill=black,circle,inner sep=1.5pt] (v21) at (8,2) {};
        \node[fill=black,circle,inner sep=1.5pt] (v22) at (10,2) {};
        \node[fill=black,circle,inner sep=1.5pt] (v23) at (12,2) {};
        \node[fill=black,circle,inner sep=1.5pt] (v41) at (14,2) {};        
        \node[fill=black,circle,inner sep=1.5pt] (v42) at (16,2) {};  
							\node[fill=black,circle,inner sep=1.5pt] (v43) at (18,2) {};

        \node[fill=black,circle,inner sep=1.5pt] (v111) at (1.5,0) {};
        \node[fill=black,circle,inner sep=1.5pt] (v112) at (2,0) {};
				 \node[fill=black,circle,inner sep=1.5pt] (v113) at (2.5,0) {};
        \node[fill=black,circle,inner sep=1.5pt] (v121) at (3.5,0) {};
        \node[fill=black,circle,inner sep=1.5pt] (v122) at (4,0) {};
				\node[fill=black,circle,inner sep=1.5pt] (v123) at (4.5,0) {};
        \node[fill=black,circle,inner sep=1.5pt] (v131) at (5.5,0) {};
        \node[fill=black,circle,inner sep=1.5pt] (v132) at (6,0) {};
				\node[fill=black,circle,inner sep=1.5pt] (v133) at (6.5,0) {};
        \node[fill=black,circle,inner sep=1.5pt] (v211) at (7.5,0) {};
        \node[fill=black,circle,inner sep=1.5pt] (v212) at (8,0) {};
				\node[fill=black,circle,inner sep=1.5pt] (v213) at (8.5,0) {};
        \node[fill=black,circle,inner sep=1.5pt] (v221) at (9.5,0) {};
        \node[fill=black,circle,inner sep=1.5pt] (v222) at (10,0) {};
        \node[fill=black,circle,inner sep=1.5pt] (v223) at (10.5,0) {};

        \draw (v)--(v1);
        \draw (v)--(v2);
        \draw (v)--(v4);
        \draw (v1)--(v11);
        \draw (v1)--(v12);
        \draw (v1)--(v13);
        \draw (v2)--(v21);
        \draw (v2)--(v22);
        \draw (v2)--(v23);
        \draw (v4)--(v41);
        \draw (v4)--(v42);
				\draw (v4)--(v43);
        \draw (v11)--(v111);
        \draw (v11)--(v112);
				\draw (v11)--(v113);
        \draw (v12)--(v121);
        \draw (v12)--(v122);
				\draw (v12)--(v123);
        \draw (v13)--(v131);
        \draw (v13)--(v132);
				\draw (v13)--(v133);
        \draw (v21)--(v211);
        \draw (v21)--(v212);
				\draw (v21)--(v213);
        \draw (v22)--(v221);
        \draw (v22)--(v222);
        \draw (v22)--(v223);
      
    \end{tikzpicture}

\caption{A rooted complete ternary tree.}
\label{fig:com}
\end{figure}
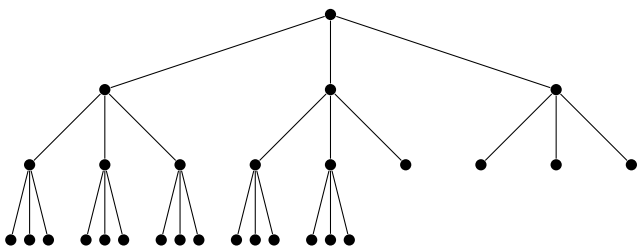

It is well known that the rooted complete $k$-ary tree is the densest among $k$-ary trees of given order, analogous to the star among general trees. We will show that such a tree also maximizes the root rank among $k$-ary trees of given order.

\begin{proposition}\label{prop:kary1}
For a proper $k$-ary tree $T$ of order $n$ and root $r$ such that $k>1$,
$$ R_T(r) \leq \left\lfloor \log_k (n(k-1)+1) \right\rfloor -1 $$
with equality if $T$ is a rooted complete $k$-ary tree.
\end{proposition}
\begin{proof}
First we claim that 
$$ R_T(r) < h \hbox{ if } n < 1+k+k^2 + \cdots + k^h = \frac{k^{h+1}-1}{k-1} . $$
Indeed, to have $R_T(r) \geq h$ requires each leaf to be at distance at least $h$ from $r$, implying that every vertex within distance $h-1$ of $r$ has $k$ children. Such a tree has at least $1+k+ \cdots +k^h$ vertices.
On the other hand, if $n \geq 1+k+k^2 + \cdots + k^{h-1} = \frac{k^{h}-1}{k-1}$, in the rooted complete $k$-ary tree every vertex within distance $h-2$ of $r$ has $k$ children, implying that $R_T(r) \geq h-1$.

Consequently we have
$$ \max R_T(r) = h-1 $$
if and only if
$$ \frac{k^{h}-1}{k-1} \leq n < \frac{k^{h+1}-1}{k-1} , $$
which is equivalent to
$$ h \leq \log_k (n(k-1)+1) < h+1 . $$
\end{proof}

\medskip
The tree transformation from $T$ to $T'$ in Proposition~\ref{prop:prop1'} preserves the entire outdegree sequence. Indeed, the only change in the vertex outdegrees comes from vertices $r$, $v$ and $w$ of $T$. Back to the proof of Proposition~\ref{prop:prop1'}, relabel $v$ as $v_1$ in the subtree $S_1$ (rooted at $r$) of $T$, and $v$ as $v_2$ in the subtree $S_2$ (rooted at $v$) of $T$. Since in the tree $T'$, vertices $w$ and $r$ are identified (say, as a new vertex $v'$), we deduce the following:
\begin{itemize}
\item $w$ is a leaf in $T$, and $v_1$ becomes a leaf in $T'$;
\item the outdegree of $r$ in $T$ equals its outdegree in $S_1$, and the outdegree of $v'$ in $T'$ equals its outdegree in $S_1$ (as a subtree of $T'$);
\item the outdegree of $v$ in $T$ equals its outdegree in $S_2$, and the outdegree of $v_2$ in $T'$ equals its outdegree in $S_2$ (as a subtree of $T'$).
\end{itemize}
Thus, we obtain:
\begin{proposition}\label{prop:proPPP}
Let $\mathcal{T}_\mathbb{S}$ denote the set of all rooted trees with a given outdegree sequence $\mathbb{S}$. Then $R_{\mathcal{T}_\mathbb{S}} = R_T(r)$ for some $T \in \mathcal{T}_\mathbb{S}$ with root $r$.
\end{proposition}
Further corollaries can follow, as eg. prescribing the number of leaves over the set of all rooted trees having no vertex of outdegree $1$ (\textit{a.k.a} topological trees or series-reduced trees).

\section{Conclusions and future work}

We have produced two different constructions of proper binary trees on a given number of leaves which obtain maximum security. Certainly there are other maximal constructions, as indicated in Proposition~\ref{prop:flips}, and it remains an open problem to classify all structures. In addition, counting the number of maximal configurations would result in an answer to the following question. 

\begin{question}
    Among all proper binary trees of order $n$, what is the probability that a uniformly chosen tree has maximum security?
\end{question}

Our study has largely been restricted to proper binary trees. The history and development of studies on the protection number also began on specific families of trees, only recently handling cases related to general families of trees. This leads to the question below. 

\begin{question}
    Let $X$ be your favourite family of rooted trees. Classify trees of type $X$ for which security is maximized. 
\end{question}

\begin{question}
   Considering Proposition~\ref{prop:path1}, can the statement be generalized to other classes of trees, such as rooted $k$-ary trees, and trees with a given segment sequence.
\end{question}

Finally, from a stochastic perspective, we present the following problem. 
\begin{question}
    Given a tree that obtains maximum security on $\ell \geq 2$ leaves, we label the $\ell$ leaves with index proportional to the height of the leaf (each leaf is numbered uniquely $x_0, x_1, \ldots, x_{\ell-1}$ with lower index allocated to lower height). ``Grow'' each $x_i$ into a complete tree of height $i$. We claim that this is often also a maximal tree.
    When is this tree maximal, when is it not?
\end{question}

As far as we are aware, the security of trees is a new approach to studying the protection number or rank of trees. Therefore, it opens up a number of future directions yet to be explored, and we hope that our results in this direction are just the tip of the research iceberg!

\acknowledgements
\label{sec:ack}
This material is based upon work supported by the National Science Foundation under Grant No. DMS 1641020, and was funded in part by the Austrian Science Fund (FWF) [10.55776/DOC78] and the Engineering and Physical Sciences Research Council (EPSRC) [EP/X03027X/1]. For the purpose of open access, the authors have applied a CC BY public copyright license to any author-accepted manuscript version arising from this submission. The authors are grateful to the reviewers for their insightful comments, which significantly improved this manuscript.


\nocite{*}
\bibliographystyle{abbrvnat}
\bibliography{sample-dmtcs}
\label{sec:biblio}

\end{document}